\documentclass[a4paper]{article}

\usepackage{authblk}
\usepackage{PRIMEarxiv}

\usepackage[utf8]{inputenc} % allow utf-8 input
\usepackage[T1]{fontenc}    % use 8-bit T1 fonts
\usepackage{hyperref}       % hyperlinks
\usepackage{url}            % simple URL typesetting
\usepackage{booktabs}       % professional-quality tables
\usepackage{amsfonts}       % blackboard math symbols
\usepackage{nicefrac}       % compact symbols for 1/2, etc.
\usepackage{microtype}      % microtypography
\usepackage{lipsum}
\usepackage{fancyhdr}       % header
\usepackage{graphicx}       % graphics
\graphicspath{{media/}}     % organize your images and other figures under media/ folder
\usepackage{relsize}
\usepackage{amsmath}
\usepackage{amsthm}
\usepackage{amssymb}
\usepackage{tikz}
\usepackage{tikz-cd}
\usetikzlibrary{calc}
\usepackage{faktor}
\usepackage{adjustbox}
\usepackage{multicol}
  \newtheorem{theorem}{Theorem}[section]
	\newtheorem{proposition}[theorem]{Proposition}
	
	\newtheorem{lemma}[theorem]{Lemma}
	\newtheorem{corollary}[theorem]{Corollary}
	\theoremstyle{definition}
	\newtheorem*{claim*}{Claim}
	\newtheorem{remark}[theorem]{Remark}
	
	\newtheorem{definition}[theorem]{Definition}
	
	\newtheorem{example}[theorem]{Example}
	\newtheorem{examples}[theorem]{Examples}
	\theoremstyle{remark}

   \DeclareMathOperator {\lk}{lk} 
   \DeclareMathOperator {\slk}{slk} 

  \DeclareMathOperator {\Ker}{Ker} 

\makeatletter
\newcommand{\subjclass}[2][1991]{%
	\let\@oldtitle\@title%
	\gdef\@title{\@oldtitle\footnotetext{#1 \emph{Mathematics subject classification.} #2}}%
}
\makeatother

\newcommand{\abs}[1]{\left|#1\right|}
\usepackage{scalerel,stackengine}
\stackMath
\newcommand\reallywidehat[1]{%
	\savestack{\tmpbox}{\stretchto{%
			\scaleto{%
				\scalerel*[\widthof{\ensuremath{#1}}]{\kern-.6pt\bigwedge\kern-.6pt}%
				{\rule[-\textheight/2]{1ex}{\textheight}}%WIDTH-LIMITED BIG WEDGE
			}{\textheight}% 
		}{0.5ex}}%
	\stackon[1pt]{#1}{\tmpbox}%
}
\newcommand\restr[2]{{% we make the whole thing an ordinary symbol
		\left.\kern-\nulldelimiterspace % automatically resize the bar with \right
		#1 % the function
		\right|_{#2} % this is the delimiter
}}
  
\title{On the $\Sigma$-invariants of Artin groups satisfying the $K(\pi,1)$-conjecture}
\author{
	{\bf Marcos Escartín-Ferrer, Conchita Martínez-Perez}
}

\affil{Department of Mathematics, University of Zaragoza}

\begin{document}

\maketitle

\begin{abstract}
We consider $\Sigma$-invariants of Artin groups that satisfy the $K(\pi,1)$-conjecture. These invariants determine the cohomological finiteness conditions of subgroups that contain the derived subgroup. We extend a known result for even Artin groups of FC-type, giving a sufficient  condition for a character $\chi:A_\Gamma\to\mathbb{R}$ to belong to  $\Sigma^n(A_\Gamma,\mathbb{Z})$. We also prove some partial converses.  As applications, we prove that the $\Sigma^1$-conjecture holds true when there is a prime $p$ that divides $l(e)/2$ for any edge with even label $l(e)>2$, we generalize to Artin groups the homological version of the Bestvina-Brady theorem and we  compute the $\Sigma$-invariants of all irreducible spherical and affine Artin groups and triangle Artin groups, which provide a complete classification of the $F_n$ and $FP_n$ properties of their derived subgroup.
\end{abstract}
\subjclass[2020]{Primary 20J06, 20F36; Secondary 57M07, 55P20}

\keywords{Artin groups \and $\Sigma$-Invariants \and Finiteness properties}

\thanks{\noindent
The authors are partially supported by Departamento de Ciencia, Universidad y Sociedad del 
Conocimiento del Gobierno de Arag{\'o}n (grant code: E22-23R: ``{\'A}lgebra y Geometr{\'i}a''),
and  by the Spanish Government PID2021-126254NB-I00}

\section{Introduction}
Given a finite simple graph $\Gamma$, i.e. a finite graph with no double edges nor loops, whose edges are labelled with positive integers $\geq 2$, the \textbf{Artin group} associated to $\Gamma$ is defined as:
$$A_\Gamma=\langle v\in V(\Gamma)\mid \langle u,v\rangle^{l(e)}=\langle v,u\rangle^{l(e)}~\forall~e=\lbrace u,v\rbrace\in E(\Gamma)\rangle$$ 
where $V(\Gamma)$ and $E(\Gamma)$  respectively denote the set of vertices and edges of $\Gamma$, $l(e)$ is the label of the edge $e$, 
$ \langle u,v\rangle^{l(e)}=(uv)^{\frac{l(e)}{2}}$ if $l(e)$ is even and $ \langle u,v\rangle^{l(e)}=(uv)^{\frac{l(e)-1}{2}}u$ if $l(e)$ is odd. This family of groups  generalizes right-angled Artin groups (RAAGs) where all the edges are labelled with a $2$. For every Artin group one can consider its associated  \textbf{Coxeter group} $W_\Gamma$, which is the quotient of $A_\Gamma$ by the normal subgroup generated by $\langle v^2\mid v\in V(\Gamma)\rangle$. It is a well-known fact that finite Coxeter groups are classified (cf. \cite{Coxeter}). There are $4$ families of irreducible finite Coxeter groups and $6$ sporadic groups, which can be represented by the Dynkin diagram. This diagram has the same number of vertices as the defining graph $\Gamma$ but in this case, two vertices are connected iff the edge connecting them in $\Gamma$ has label $\geq 3$ and, for simplicity, one omits the label in the edges labelled with a $3$:
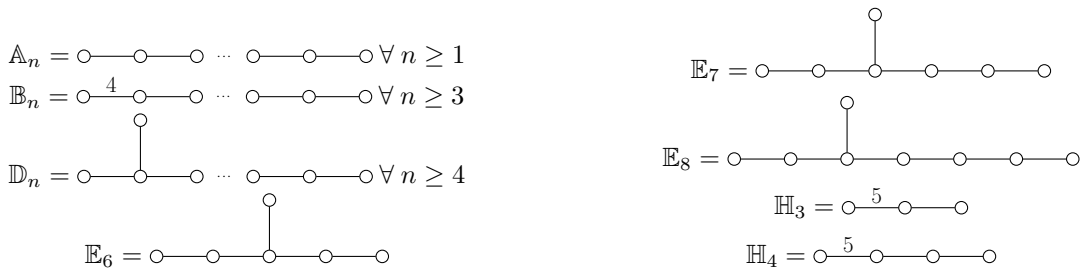
\begin{multicols}{2}
	 $$\mathbb{A}_n=\begin{tikzpicture}[main/.style = {draw, circle},node distance={15mm},scale=0.6, every node/.style={scale=0.5}]
	\node[main] (1) {};
	\node[main] (2) [right of=1] {}; 
	\node[main] (3) [right of=2] {}; 
	\draw[-] (1) --  (2);
	\draw[-] (2) -- (3);
	\node[main] (4) [right of=3] {};
	\node[main] (5) [right of=4] {}; 
	\node[main] (6) [right of=5] {}; 
	\draw[-] (4) -- (5);
	\draw[-] (5) -- (6);
	\node at ($(3)!.5!(4)$) {\ldots};
\end{tikzpicture}~\forall~n\geq1$$
$$\mathbb{B}_n=
\begin{tikzpicture}[main/.style = {draw, circle},node distance={15mm},scale=0.6, every node/.style={scale=0.5}] 
	\node[main] (1) {};
	\node[main] (2) [right of=1] {}; 
	\node[main] (3) [right of=2] {}; 
	\draw[-] (1) -- node[above] {$\mathlarger{\mathlarger{\mathlarger{\mathlarger{4}}}}$}   (2);
	\draw[-] (2) -- (3);
	\node[main] (4) [right of=3] {};
	\node[main] (5) [right of=4] {}; 
	\node[main] (6) [right of=5] {}; 
	\draw[-] (4) --  (5);
	\draw[-] (5) -- (6);
	\node at ($(3)!.5!(4)$) {\ldots};
\end{tikzpicture}~\forall~n\geq3$$
$$\mathbb{D}_n=
\begin{tikzpicture}[main/.style = {draw, circle},node distance={15mm},scale=0.6, every node/.style={scale=0.5}] 
	\node[main] (1) {};
	\node[main] (2) [right of=1] {}; 
	\node[main] (3) [right of=2] {}; 
	\node[main] (4) [above of=2] {}; 
	\draw[-] (1) -- (2);
	\draw[-] (2) -- (3);
	\draw[-] (2) -- (4);
	\node[main] (5) [right of=3] {};
	\node[main] (6) [right of=5] {}; 
	\node[main] (7) [right of=6] {}; 
	\draw[-] (5) --  (6);
	\draw[-] (6) -- (7);
	\node at ($(3)!.5!(5)$) {\ldots};
\end{tikzpicture}~\forall~n\geq4$$
$$\mathbb{E}_6=
\begin{tikzpicture}[main/.style = {draw, circle},node distance={15mm},scale=0.6, every node/.style={scale=0.5}] 
	\node[main] (1) {};
	\node[main] (2) [right of=1] {}; 
	\node[main] (3) [right of=2] {}; 
	\node[main] (4) [above of=3] {}; 
	\node[main] (5) [right of=3] {}; 
	\node[main] (6) [right of=5] {}; 
	\draw[-] (1) -- (2);
	\draw[-] (2) -- (3);
	\draw[-] (3) -- (4);
	\draw[-] (3) -- (5);
	\draw[-] (5) -- (6);
\end{tikzpicture}$$
$$\mathbb{E}_7=
\begin{tikzpicture}[main/.style = {draw, circle},node distance={15mm},scale=0.6, every node/.style={scale=0.5}] 
	\node[main] (1) {};
	\node[main] (2) [right of=1] {}; 
	\node[main] (3) [right of=2] {}; 
	\node[main] (4) [above of=3] {}; 
	\node[main] (5) [right of=3] {}; 
	\node[main] (6) [right of=5] {}; 
	\node[main] (7) [right of=6] {}; 
	\draw[-] (1) -- (2);
	\draw[-] (2) -- (3);
	\draw[-] (3) -- (4);
	\draw[-] (3) -- (5);
	\draw[-] (5) -- (6);
	\draw[-] (6) -- (7);
\end{tikzpicture}$$
$$\mathbb{E}_8=
\begin{tikzpicture}[main/.style = {draw, circle},node distance={15mm},scale=0.6, every node/.style={scale=0.5}] 
	\node[main] (1) {};
	\node[main] (2) [right of=1] {}; 
	\node[main] (3) [right of=2] {}; 
	\node[main] (4) [above of=3] {}; 
	\node[main] (5) [right of=3] {}; 
	\node[main] (6) [right of=5] {}; 
	\node[main] (7) [right of=6] {}; 
	\node[main] (8) [right of=7] {}; 
	\draw[-] (1) -- (2);
	\draw[-] (2) -- (3);
	\draw[-] (3) -- (4);
	\draw[-] (3) -- (5);
	\draw[-] (5) -- (6);
	\draw[-] (6) -- (7);
	\draw[-] (7) -- (8);
\end{tikzpicture}$$
$$\mathbb{H}_3=
\begin{tikzpicture}[main/.style = {draw, circle},node distance={15mm},scale=0.6, every node/.style={scale=0.5}] 
	\node[main] (1) {};
	\node[main] (2) [right of=1] {}; 
	\node[main] (3) [right of=2] {}; 
	\draw[-] (1) -- node[above] {$\mathlarger{\mathlarger{\mathlarger{\mathlarger{5}}}}$} (2);
	\draw[-] (2) -- (3);
\end{tikzpicture}$$
$$\mathbb{H}_4=
\begin{tikzpicture}[main/.style = {draw, circle},node distance={15mm},scale=0.6, every node/.style={scale=0.5}] 
	\node[main] (1) {};
	\node[main] (2) [right of=1] {}; 
	\node[main] (3) [right of=2] {}; 
	\node[main] (4) [right of=3] {}; 
	\draw[-] (1) -- node[above] {$\mathlarger{\mathlarger{\mathlarger{\mathlarger{5}}}}$} (2);
	\draw[-] (2) -- (3);
	\draw[-] (3) -- (4);
\end{tikzpicture}$$
$$\mathbb{F}_4=
\begin{tikzpicture}[main/.style = {draw, circle},node distance={15mm},scale=0.6, every node/.style={scale=0.5}] 
	\node[main] (1) {};
	\node[main] (2) [right of=1] {}; 
	\node[main] (3) [right of=2] {}; 
	\node[main] (4) [right of=3] {}; 
	\draw[-] (1) -- (2);
	\draw[-] (2) -- node[above] {$\mathlarger{\mathlarger{\mathlarger{\mathlarger{4}}}}$} (3);
	\draw[-] (3) -- (4);
\end{tikzpicture}$$
$$\mathbb{I}_2(k)=
\begin{tikzpicture}[main/.style = {draw, circle},node distance={15mm},scale=0.6, every node/.style={scale=0.5}] 
	\node[main] (1) {};
	\node[main] (2) [right of=1] {}; 
	\draw[-] (1) -- node[above] {$\mathlarger{\mathlarger{\mathlarger{\mathlarger{k}}}}$} (2);
\end{tikzpicture}~\forall~k\geq 4$$
\end{multicols}
Every finite Coxeter group is a finite direct product of copies of groups of this list, where each component is given by one connected component of the Dynkin diagram.\\
Some families of Artin groups we will be considering are the \textbf{even Artin groups} (all the edges are labelled with an even number), the \textbf{spherical (or finite type) Artin groups} (the associated Coxeter group is finite) and the \textbf{Artin groups of FC-type} (for every complete subgraph $X\subset\Gamma$, also called clique, the group $A_X$, which coincides with the subgroup of $A_\Gamma$ generated by the $v\in X$, is a spherical Artin group). Observe that the classification of irreducible finite Coxeter groups gives a classification of spherical irreducible Artin groups. A very useful property of Artin groups is that often we can find combinatorial properties of the graph $\Gamma$ that characterize group theoretical properties of $A_\Gamma$.\\
The most important open problem on Artin groups is the so-called $K(\pi,1)$-conjecture. One of the equivalent formulations of this conjecture states the existence of an explicit finite model for a classifying space for $A_\Gamma$ called the Salvetti complex. The $K(\pi,1)$-conjecture is known to be true for many important families of Artin groups:
\begin{itemize}
	\item $A_\Gamma$ of \textbf{large type}, i.e. all the edges of $\Gamma$ are labelled with integers $\geq3$. (\cite{Hendriks} Hendriks)
	\item $A_\Gamma$ of \textbf{dimension 2}, i.e. if $X\leq\Gamma$ is spherical then $\abs{X}=2$. (\cite{Charney-Davis} Charney-Davis)
	\item $A_\Gamma$ of FC-type. (\cite{Charney-Davis} Charney-Davis)
	\item $A_\Gamma$ of \textbf{affine type}, i.e. $W_\Gamma$ acts as a proper, cocompact group of isometries on some Euclidean space with  the vertices of $\Gamma$ acting as affine reflections. (\cite{Paolini-Salvetti} Paolini-Salvetti)
\end{itemize}
Moreover, if $A_\Gamma$ satisfies the $K(\pi,1)$-conjecture, then $A_X$ satisfies the $K(\pi,1)$-conjecture for any induced subgraph $X\subset\Gamma$ (cf. \cite{Godelle-Paris}) and if $A_X$ satisfies the $K(\pi,1)$-conjecture for every complete $X\subset\Gamma$, then $A_\Gamma$ also satisfies the $K(\pi,1)$-conjecture (cf. \cite{Ellis}). For the reader interested in further details, a comprehensive summary of the conjecture was written by Paris in \cite{Paris}. In this paper we will be using two properties that are implied by this conjecture: if the $K(\pi,1)$-conjecture holds for the Artin group $A_\Gamma$, then the Deligne complex (sometimes called modified Deligne complex) is contractible and the Salvetti complex is a model for $K(\pi,1)$, the proofs of these facts can be found in \cite{Charney-Davis} and \cite{Paris} (for definitions and more details, see Sections \ref{Section 2.2} and \ref{Section 6.1}).\\
 Let $A$ be a left $G$-module. The $\Sigma$-invariants of a finitely generated group $G$ are sets $\Sigma^n(G)$ and $\Sigma^n(G,A)$ of equivalence classes of characters $\chi:G\to\mathbb{R}$. These invariants encode information that allows one to determine which subgroups $G'\leq N\leq G$ have homotopical type $F_n$ or homological type $FP_n$. In general, the computation of the  $\Sigma$-invariants of a group is extremely difficult, but there are some cases in which a full description is available.\\

In this paper, we will state a sufficient condition in terms of combinatorial properties of the defining graph for a character to be in $\Sigma^n(G)$ or in $\Sigma^n(G,\mathbb{Z})$ in the case where $G=A_\Gamma$ is an Artin group that satisfies the $K(\pi,1)$-conjecture. This generalizes one of the main results given in \cite{Blasco-Cogolludo-Martinez} for the case when $A_\Gamma$ is even of FC-type.
%, where the second named author together with Blasco and Cogolludo gave a combinatorial condition depending on the graph $\Gamma$ that implies for a character $\chi:A_\Gamma\to\mathbb{R}$ that $[\chi]$ %lies in (some of) the invariants. 
To state this generalization, in section \ref{Strong} we will define what we call the {\bf strong $n$-link condition} and its homotopical version. Then we prove:
\begin{theorem}\label{Theorem 1.2}
	Let $A_\Gamma$ be an Artin group satisfying the $K(\pi,1)$-conjecture and let $0\neq\chi:A_\Gamma\to\mathbb{R}$ be a character. Then
	\begin{enumerate}
		\item If $\chi$ satisfies the strong $n$-link condition then $[\chi]\in\Sigma^n(A_\Gamma,\mathbb{Z})$,
		\item If $\chi$ satisfies the strong homotopical $n$-link condition then $[\chi]\in\Sigma^n(A_\Gamma)$.
	\end{enumerate}
\end{theorem}
We do not know whether the converse of this result holds true in general, but to prove some partial converses we will construct a spectral sequence, using a filtration of the cyclic cover of the Salvetti complex associated to a given discrete character $\chi:A_\Gamma\to\mathbb{Z}$. The kernel of such a  character  $A_\Gamma^\chi:=\Ker(\chi)$ is called an \textbf{Artin kernel}, these groups generalize the family of Bestvina-Brady groups, which are kernels of the character given by sending all the standard generators to $1$.

This spectral sequence (Proposition \ref{Proposition 8.3}) allows one to compute the infinite dimensional part of the homology of $\Ker(\chi)$ with coefficients in a field. Moreover, we show that under certain conditions the sequence happens to collapse at page $1$. Using this we are able to prove the converse of Theorem \ref{Theorem 1.2} in the following cases:
\begin{itemize}
	\item[i)] $\chi$ is discrete and there is no $\emptyset\neq\Delta\subset\Gamma$ spherical with $\chi(Z(A_\Delta))=0$.
	\item[ii)] $\chi$ is discrete, $\exists$ a prime number $p$ with $p\mid\frac{l(e)}{2}~\forall~e\in\Gamma$ with $l(e)>2$ even, and $\chi(v)\neq0~\forall~v\in\Gamma$.
	\item[iii)] $\Gamma$ is even and $\exists$ a prime number $p$ with $p\mid\frac{l(e)}{2}~\forall~e\in\Gamma$ with $l(e)>2$.
\end{itemize}

In particular, from i) we can extend the homological version of the celebrated Bestvina-Brady theorem to arbitrary Artin groups that satisfy the $K(\pi,1)$-conjecture.

\begin{theorem}\label{BestvinaBradyArtin}
	Let $A_\Gamma$ be an Artin group satisfying the $K(\pi,1)$-conjecture and $B_\Gamma$ the kernel of the map $0\neq\chi:A_\Gamma\to\mathbb{Z}$ given by $\chi(v)=1$ for every $v\in\Gamma$. Then $B_\Gamma$ is of homological type $FP_n$ if and only if the complex obtained from $\Gamma$ by adding an $(m-1)$-cell to every spherical subset $X\subseteq\Gamma$ with $\abs{X}=m$ is $(n-1)$-acyclic
	\end{theorem}

Among all the Sigma-invariants of a group, the first one, $\Sigma^1$ is considered to be the most important one. In the case of a right-angled Artin group, $\Sigma^1$ admits a combinatorial description in terms of the so called \textbf{living subgraph}, which is the graph obtained from $\Gamma$ after removing all vertices with $\chi$-value 0, called \textbf{dead vertices}. The notion of living subgraph can be extended to  arbitrary Artin groups as follows.
An edge $e=\lbrace u,v\rbrace\in \Gamma$ is said to be \textbf{dead} if $l(e)>2$ is even and $\chi(u)+\chi(v)=0$. The living subgraph is the subgraph $\mathrm{Liv}^\chi$ of $\Gamma$ obtained after removing all dead vertices and the interior of all dead edges. In the $\Sigma^1$ case, following some work of Meier-Meinert-VanWyk from \cite{Meier-Meiner-VanWyk}, it was conjectured in \cite{Almeida} that a character $\chi:A_\Gamma\to\mathbb{R}$ belongs to $\Sigma^1(A_\Gamma)$ if and only if $\mathrm{Liv}^\chi$ is connected and dominant, where {\bf dominant} means that every vertex of $\Gamma\setminus\mathrm{Liv}^\chi$ is connected to some vertex of $\mathrm{Liv}^\chi$. For more details about this conjecture see section \ref{Strong}. As a consequence of the partial converses of Theorem \ref{Theorem 1.2} we can prove the $\Sigma^1$-conjecture for a family of Artin groups.
\begin{theorem}\label{Theorem 1.4}
	Let $A_\Gamma$ be an Artin group such that there exists a prime number $p$ with $p\mid\frac{l(e)}{2}$ for any edge $e$ in $\Gamma$ with even label $l(e)>2$. Then $[\chi]\in\Sigma^1(A_\Gamma,\mathbb{Z})\Leftrightarrow\text{ Liv}^\chi$ is connected and dominant.
\end{theorem}

Applying our results, one can easily compute   the $\Sigma$-invariants of all irreducible spherical Artin groups. Indeed, if we denote  by $A_S$  the Artin group of type $S$,
we prove:

\begin{theorem}\label{Theorem 1.1}
	Let $G$ be an irreducible spherical Artin group. Then:
	\begin{enumerate}
		\item If $G=A_{\mathbb{I}_2(k)}$ with $k$ even, then $S(G)=\mathbb{S}^1$ and: 
		$$\Sigma^j(G)=\Sigma^j(G,\mathbb{Z})=S(G)\setminus\lbrace\pm[(1,-1)]\rbrace\text{ for }j\geq 1$$
		\item If $G=A_{\mathbb{F}_4}$, then $S(G)=\mathbb{S}^1$ and:
		 $$\Sigma^j(G)=\Sigma^j(G,\mathbb{Z})=S(G)\text{ for }j=0,1,2$$
		 $$\Sigma^j(G)=\Sigma^j(G,\mathbb{Z})=S(G)\setminus\lbrace\pm[(1,-1)]\rbrace\text{ for }j\geq 3$$
		\item If $G=A_{\mathbb{B}_n}$ ($n$ is the number of generators), then $S(G)=\mathbb{S}^1$ and:
		$$\Sigma^j(G)=\Sigma^j(G,\mathbb{Z})=S(G)\text{ for }j=0,\dots,n-2$$
		$$\Sigma^j(G)=\Sigma^j(G,\mathbb{Z})=S(G)\setminus\lbrace\pm[(n-1,-1)]\rbrace\text{ for }j\geq n-1$$
		\item  $S(G)=\mathbb{S}^0$ and $\Sigma^j(G)=\Sigma^j(G,\mathbb{Z})=S(G)\text{ for }j\geq 0$ in all other cases.
	\end{enumerate}
	In the first three cases $(a,b)$ represents the character given by $\chi(v)=a$ and $\chi(w)=b$, where $v$ and $w$ represent the two generators of $G/G'\cong\mathbb{Z}^2$.
\end{theorem}
The case when $G=A_{\mathbb{I}_2(k)}$ with $k$ even had been previously computed in \cite[Page 76]{Meier-Meiner-VanWyk}, but note that in the case when $G=A_{\mathbb{B}_3}$ our computations below show that $\Sigma^2(A_{\mathbb{B}_3})=S(A_{\mathbb{B}_3})\setminus\lbrace\pm[(2,-1)]\rbrace$ instead of $S(A_{\mathbb{B}_3})\setminus\lbrace\pm[(1,-1)]\rbrace$ as stated in \cite[Page 76]{Meier-Meiner-VanWyk}.
As a corollary, we obtain information on the finiteness properties of the derived subgroup of the spherical irreducible Artin groups. 
\begin{corollary}\label{Corollary 1.2}
	Let $G$ be a spherical irreducible Artin group. Then:
	\begin{enumerate}
		\item If $G=A_{\mathbb{I}_2(k)}$ with $k$ even then $G'$ is not finitely generated.
		\item If $G=A_{\mathbb{F}_4}$ then $G'$ is $F_2$ but not $FP_3$.
		\item If $G=A_{\mathbb{B}_m}$ then $G'$ is $F_{n-2}$ but not $FP_{n-1}$.
		\item  $G'$ is $F_\infty$ in all other cases.
	\end{enumerate}
\end{corollary}

Finally, in the last section, we apply Theorem \ref{Theorem 1.2} and the partial converses to compute the $\Sigma$-invariants of irreducible Artin groups of affine type and of triangle Artin groups.

This paper is structured as follows. In section \ref{Section 2} we will recall some of the previous results that will be used throughout the paper. Section \ref{Section 3} is devoted to the study of the centre of Artin groups. 
 In section \ref{Strong} we will give an overview of some of the conditions given in the literature, which will motivate the definition of the strong $n$-link condition. In section \ref{Section 5} we will prove Theorem \ref{Theorem 1.2} and its homotopical analogue. Then in section \ref{Section 6} we use the Salvetti complex to construct a spectral sequence that allows one to detect infinite dimensional bits in the homology groups of kernels of discrete characters with coefficients in a field. This spectral sequence is used in section \ref{Section 7} to find some partial converses to Theorem \ref{Theorem 1.2}. As applications, we will deduce Theorems \ref{BestvinaBradyArtin}, \ref{Theorem 1.4} and \ref{Theorem 1.1}. Moreover, we will use our results to compute the $\Sigma$-invariants of  affine  and triangle Artin groups.

\section{Preliminaries}\label{Section 2}
\subsection{$\Sigma$-invariants}
Let $G$ be finitely generated. Throughout the paper  $\chi:G\to\mathbb{R}$ will represent an arbitrary non-zero character. Two characters are said to be \textbf{equivalent}, denoted $\chi_1\sim\chi_2$, if one is a positive scalar multiple of the other, i.e. $\exists~t>0$ such that $\chi_1=t\chi_2$. The \textbf{character sphere} $S(G)$ will be the set of equivalence classes of characters.  Observe that $S(G)\cong \mathbb{S}^{r-1}$, where $r$ is the torsion free rank of the abelianization $G/G'$ and $\mathbb{S}^{r-1}$ is the $r-1$-sphere. A character $\chi:G\to\mathbb{Z}$ is called \textbf{discrete} and the set of all discrete characters is dense in $S(G)$ (cf. \cite{Strebel} Lemma B3.24). We define next the homological $\Sigma$-invariants of a group.
\begin{definition}
	Consider a character $\chi:G\to\mathbb{R}$ and $G_\chi=\lbrace g\in G\mid\chi(g)\geq 0\rbrace$, which is a monoid. Then, if $A$ is a left $G$-module and $n$ a positive integer, we define:
	$$\Sigma^n(G,A)=\lbrace[\chi]\in S(G)\mid A\text{ is of type }FP_n\text{ over }\mathbb{Z}G_\chi\rbrace$$
	Observe that:
	$$\Sigma^\infty(G,A)\subseteq\dots\subseteq\Sigma^1(G,A)\subseteq\Sigma^0(G,A)=S(G)$$
\end{definition}
There exists also a homotopical version: 

\begin{definition}
	Let $G$ be a group of type $F_n$ and $X$ a model for $K(G,1)$ with finite $n$-skeleton and only one vertex. Denote  by $\widetilde{X}$  the universal cover of $X$. There is a bijective correspondence between vertices of $\widetilde{X}$ and elements of $G$. If $\chi:G\to\mathbb{R}$ is a character one can define $\widetilde{\chi}:\widetilde{X}\to\mathbb{R}$  by $\widetilde{\chi}(v)=\chi(g)$ for vertices $v=g\cdot b$ where $b$ is some fixed base vertex and extend linearly and $G$-equivariantly from the vertices to the entire universal cover. For $a\in\mathbb{R}$, let $\widetilde{X_{\chi}}^{[a,\infty)}$ be the maximal subcomplex in
	 $\lbrace x\in\widetilde{X}\mid \widetilde{\chi}(x)\geq a\rbrace$, then $\forall~d\geq 0$ there is an inclusion map $\widetilde{X_{\chi}}^{[0,\infty)}\hookrightarrow \widetilde{X_{\chi}}^{[-d,\infty)}$ which induces a map $\pi_i\left(\widetilde{X_{\chi}}^{[0,\infty)}\right)\rightarrow \pi_i\left(\widetilde{X_{\chi}}^{[-d,\infty)}\right)$. The $n$-th homotopical $\Sigma$-invariant is defined as:
	$$\Sigma^n(G)=\left\lbrace[\chi]\in S(G)\mid\exists~d>0\text{ such that } \pi_i\left(\widetilde{X_{\chi}}^{[0,\infty)}\right)\rightarrow \pi_i\left(\widetilde{X_{\chi}}^{[-d,\infty)}\right) \text{ is trivial }\forall~i<n\right\rbrace$$
	Observe that:
	$$\Sigma^\infty(G)\subseteq\dots\subseteq\Sigma^1(G)\subseteq\Sigma^0(G)=S(G)$$
\end{definition}
One can check that this definition does not depend on the chosen model for $K(G,1)$. Note that we have defined $\Sigma^n(G,\mathbb{Z})$ and $\Sigma^n(G)$ only in the case when $G$ is of type $FP_n$ or $F_n$ respectively. However, when $G$ is not $FP_n$ or $F_n$ we will assume by convention that $\Sigma^n(G,\mathbb{Z})=\emptyset$ and $\Sigma^n(G)=\emptyset$ respectively. Moreover, if $n=\operatorname{cd}G$ where $\operatorname{cd}G$ is the cohomological dimension of $G$, then $\Sigma^m(G,\mathbb{Z})=\Sigma^n(G,\mathbb{Z})$ for any $m\geq n$, analogously, if $n=\operatorname{gd}G$ where $\operatorname{gd}G$ is the geometric dimension of $G$, then $\Sigma^m(G)=\Sigma^n(G)$ for any $m\geq n$.
Another important topological fact about these invariants is that $\Sigma^n(G,\mathbb{Z}),\Sigma^n(G)$ are open in $S(G)$ for every finitely generated group $G$ (cf. \cite{Bieri-Renz}).\\

The homotopical and homological $\Sigma$-invariants are closely related. Indeed,  $\Sigma^n(G)=\Sigma^n(G,\mathbb{Z})\cap\Sigma^2(G)$ for $n>1$ and for any commutative ring $R$, $\Sigma^n(G)\subset\Sigma^n(G,\mathbb{Z})\subset\Sigma^n(G,R)$, $\Sigma^1(G)=\Sigma^1(G,R)$   (cf. \cite{Renz}).
 The reason why the $\Sigma$-invariants encode the finiteness properties of the derived subgroup is the following central theorem.
\begin{theorem}[\cite{Bieri-Renz} Bieri-Renz]\label{Theorem 2.3}
	Let $G$ be a group of type $F_n$ and $G'\leq N\leq G$. Then:
	\begin{enumerate}
		\item  $N$ is of type $FP_n$ if and only if $\lbrace[\chi]\in S(G)\mid\chi(N)=0\rbrace\subset \Sigma^n(G,\mathbb{Z})$.
		\item  $N$ is of type $F_n$ if and only if $\lbrace[\chi]\in S(G)\mid\chi(N)=0\rbrace\subset \Sigma^n(G)$.
	\end{enumerate}
\end{theorem}
 A key result we will be using to compute the $\Sigma$-invariants of the spherical irreducible Artin groups is the following.
\begin{lemma}[\cite{Meier-Meiner-VanWyk 2} Lemma 2.1]\label{Lemma 2.4}
	Let $G$ be a group of type $F_n$ and $\chi:G\to\mathbb{R}$ a character with $\chi(Z(G))\neq 0$. Then $[\chi]\in \Sigma^n(G)\subset \Sigma^n(G,\mathbb{Z})$.
\end{lemma}
Moreover, as in \cite{Blasco-Cogolludo-Martinez}, \cite{Meier-Meiner-VanWyk} and \cite{Meier-Meiner-VanWyk 2}, many of our results will be based on  the next theorem, together with the action on the Deligne complex.\begin{theorem}[\cite{Meier-Meiner-VanWyk} Theorem 3.2]\label{Theorem 2.5}
	Let $G$ be a group of type $F_\infty$ acting by cell-permuting homeomorphisms on a $CW$-complex $X$ with finite $n$-skeleton modulo $G$. Let $\chi:G\to\mathbb{R}$ be a character such that for any $0\leq p\leq n$ and any $p$-cell $\sigma$ of $X$ the stabilizer $G_\sigma$ is not in $\Ker(\chi)$. Then:
	\begin{itemize}
		\item If $X$ is $(n-1)$-connected and $[\restr{\chi}{G_\sigma}]\in\Sigma^{n-p}(G_\sigma)$ for any $p$-cell $\sigma$ with $0\leq p\leq n$ then $[\chi]\in \Sigma^n(G)$.
		\item If $X$ is $(n-1)-R$-acyclic and $[\restr{\chi}{G_\sigma}]\in\Sigma^{n-p}(G_\sigma,R)$ for any $p$-cell $\sigma$ with $0\leq p\leq n$ then $[\chi]\in \Sigma^n(G,R)$.
	\end{itemize}
\end{theorem}
The following easy lemma from \cite{Blasco-Cogolludo-Martinez}  will be helpful later.
\begin{lemma}\label{Lemma 2.7}
	If $A_\Gamma$ is an Artin group then $[\chi]\in \Sigma^n(A_\Gamma,R)\iff [-\chi]\in\Sigma^n(A_\Gamma,R)$ for every commutative ring $R$ and $[\chi]\in \Sigma^n(A_\Gamma)\iff [-\chi]\in\Sigma^n(A_\Gamma)$.
\end{lemma}	
\begin{proof}
		Using the symmetry of the defining  relators of $A_\Gamma$ it is obvious that the map $\varphi: A_\Gamma\to A_\Gamma$ given by $\varphi(v)=v^{-1}~\forall~v\in V(\Gamma)$ induces an automorphism. Since $\chi\circ\varphi=-\chi$ and  both $\Sigma^n(A_\Gamma)$ and $\Sigma^n(A_\Gamma,R)$ are invariant under automorphisms the result follows.
	\end{proof}
In the next sections we will consider several subgraphs and simplicial complexes associated to a graph $\Gamma$ and a character $\chi:A_\Gamma\to\mathbb{R}$. One of them is the {\bf living subgraph} $\text{Liv}^\chi_0$, which is the subgraph induced by those $v\in\Gamma$ such that $\chi(v)\neq 0$.  This subgraph satisfies the following result about $\Sigma$-invariants: 
\begin{proposition}\label{Proposition 2.7}
	If $A_\Gamma$ is an Artin group and $\chi:A_\Gamma\to\mathbb{R}$ is a character such that $[\chi]\in\Sigma^n(A_\Gamma)$ then $[\chi]\in\Sigma^n(A_{\text{Liv}^\chi_0})$.
\end{proposition}
	\begin{proof}
		There is an obvious retraction $\rho:A_\Gamma\to A_{\text{Liv}^\chi_0}$ given by $v\mapsto v$ if $\chi(v)\neq0$ and $v\mapsto 1$ if $\chi(v)=0$. It is  well defined  because the edges connecting vertices in $\Gamma\setminus \text{Liv}^\chi_0$ to vertices in $\text{Liv}^\chi_0$ are all labelled with even numbers. Since $\rho$ splits it induces a semidirect product $A_\Gamma=\Ker(\rho)\rtimes A_{\text{Liv}^\chi_0}$ and applying Theorem 5.3 from \cite{Meier-Meiner-VanWyk} we are done. 
	\end{proof}

Our main tool to deduce that some discrete characters are not in certain  $\Sigma$-invariants will be the following observation that follows from Theorem \ref{Theorem 2.3} and Lemma \ref{Lemma 2.7}.
\begin{proposition}\label{Proposition 2.4}
	Let $G$ be an Artin group that satisfies the $K(\pi,1)$-conjecture, $0\neq\chi:G\to\mathbb{Z}$ and $\mathbb{F}$ a field. If $\dim_\mathbb{F} H_n(\Ker(\chi),\mathbb{F})=\infty$ then  $[\chi]\notin\Sigma^n(G,\mathbb{Z})$.
\end{proposition}

\subsection{Coset Posets}\label{Section 2.2}
Let $\mathcal{P}$ be a partially ordered set, i.e. a poset. One can define the \textbf{derived poset} 
whose elements are totally ordered chains:
	$$\Delta:(a_0<a_1<\dots<a_n)$$
	with the order $\Delta'<\Delta$ if $\Delta'$ is a subchain of $\Delta$. This is an abstract simplicial complex, so we can consider its geometric realization that we will denote as $\abs{\mathcal{P}}$. A \textbf{face} of $\Delta$ is just any subchain of $a_0<a_1<\dots<a_n$. In particular, assume that $\mathcal{P}$ is a poset of subgroups of a given group $G$. Associated to $\mathcal{P}$ we also have
	the so called \textbf{coset poset}. It is defined as the poset $G\mathcal{P}=\lbrace gS\mid g\in G,S\in\mathcal{P}\rbrace$. 
	We will   denote by $C(\mathcal{P})$  the geometric realization of the coset poset $G\mathcal{P}$. The $k$-simplices of $C(\mathcal{P})$ are:
	$$\sigma: g(S_0\subset S_1\subset\cdots\subset S_k)$$
	for some $g\in G$ and $S_0,\dots, S_k\in\mathcal{P}$. In particular, if $\mathcal{H}\subset\mathcal{P}$ is a subposet, $C(\mathcal{H})$ is the subcomplex of $C(\mathcal{P})$ formed by $k$-simplices of the form:
	$$\sigma: g(S_0\subset\dots\subset S_k)$$
	where $g\in G$ and $S_0,\dots,S_k\in\mathcal{H}$. One natural question to ask is which is the relation between the connectivity of $C(\mathcal{H})$ and of $C(\mathcal{P})$. A very useful tool to study this problem is Morse theory. 
	
	Let $X$ be a simplicial complex with a subcomplex $Y\subseteq X$. Assume there is a \textbf{height function}  $h:X^{(0)}\setminus Y^{(0)}\to\mathbb{Z}_{\geq 0}$. Moreover, we will assume that $h$ is bounded and satisfies the following: for any 1-simplex $e=\{x,y\}\subseteq X$, $x,y\not\in Y$ implies $h(x)\neq h(y)$. We can extend $h$ to $X^{(0)}$ by putting $h(y)=-1$ for any $y\in Y^{(0)}$ and then extend to a function from the face poset (i.e., the poset of simplices of $X$) to $\mathbb{Z}_{\geq-1}$ given by
	$$h(\sigma)=\max\{h(x)\mid x\text{ vertex of }\sigma\}$$
	for $\sigma$ a simplex of $X$. 
	Obviously, if $\tau\subset\sigma$, $h(\tau)\leq h(\sigma)$ so we have a finite filtration of subcomplexes
	$$Y=F^{-1}X\subseteq F^0X\subseteq\cdots\subseteq F^tX=X$$
	where $t$ is some big enough integer and
	$$F^sX=\{\text{subcomplex of cells $\sigma$ with }h(\sigma)\leq s\}.$$
	Let $y\in X^{(0)}$ be a vertex. The link (or descending link) $\lk_{F^{s-1}X}(y)$ of $y$ in $F^{s-1}X$ is the simplicial complex consisting on those simplices $\tau\in F^{s-1}X$ such that $y\not\in\tau$ but $\{y\}\cup\tau$ is a simplex of $X$. In this context, we have the following version of the Morse Lemma (\cite{Bestvina-Brady} Corollary 2.6):
	
\begin{lemma}[Relative Morse lemma] Let $Y\subseteq X$ and $h$ be as before and assume that for any $0\leq s$ and $y\in X^{(0)}$ with $h(y)=s$, the link $\lk_{F^{s-1}X}(y)$ is $(n-1)$-connected. Then the inclusion $Y\subseteq X$ induces an isomorphism in the homotopy groups $\pi_j$, $j<n$ and an epimorphism in $\pi_n$.
\end{lemma}	
\begin{proof} Note that using the Morse Lemma, each 
inclusion $F^{s-1}X\subseteq F^{s}X$ induces an isomorphism in $\pi_j$ for $j<n$ and an epimorphism in $\pi_n$.
\end{proof}

We will apply this result to certain coset posets as follows. Let $G$ be a group, $\mathcal{P}$ a poset of subgroups of $G$ and $\mathcal{H}$ a subposet of $\mathcal{P}$. Assume we have a bounded height function $h:\mathcal{P}\setminus\mathcal{H}\to\mathbb{Z}_{\geq 0}$ such that for $A\subsetneq B$, $A,B\in\mathcal{P}$ we have $h(A)<h(B)$. We may extend this function to the associated coset posets in the obvious way and get a height function for the simplicial complex $X=C(\mathcal{P})$ and a filtration that can be used to relate the connectivity of $X$ with that of $Y=C(\mathcal{H})$. 

Now, fix $S\in\mathcal{P}\setminus\mathcal{H}$ such that $h(S)=s$ and $g\in G$, we see $y=gS$ as a 0-simplex in $X\setminus Y$. In fact, translating if necessary we may assume that $g=1$. Then, the link $\lk_{F^{s-1}X}(y)$ consists of simplices of the form
$$\tau:g_0S_0\subset g_1S_1\subset\dots\subset g_{i-1}S_{i-1}\subset S_{i+1}\subset\dots\subset S_k$$
such that 
$$g_0S_0\subset g_1S_1\subset\dots\subset g_{i-1}S_{i-1}\subset S\subset S_{i+1}\subset\dots\subset S_k$$

is a simplex in $X$ and $S_j\in\mathcal{H}~\forall j>i$, because otherwise $\tau$ would not lie in $F^{s-1}X$. This means that $\lk_{F^{s-1}X}(y)$ is the join of the following two simplicial complexes:
\begin{itemize}
	\item The geometric realization of the $S$-coset poset associated to $\lbrace T\in\mathcal{P}\mid T\lneq S\rbrace$ (i.e., seen as a poset of subgroups in $S$). We denote this geometric realization by  $\mathrm{Del}^S$.
	\item The geometric realization of the poset $\mathcal{J}_S=\lbrace T\in\mathcal{H}\mid S\subset T\rbrace$ (note that this is not a coset poset). 
\end{itemize}
In other words, we have
$$\lk_{F^{s-1}X}(y)=\mathrm{Del}^S\star\abs{\mathcal{J}_S}$$
where $\star$ denotes the join. Therefore the Morse Lemma together with its homological version (see the proof of \cite[Lemma 3.9 ]{Blasco-Cogolludo-Martinez}) imply:

\begin{lemma}\label{Lemma 2.9} Let $G$ be a group and $\mathcal{H}\subseteq\mathcal{P}$ posets of subgroups of $G$.  With the same notation above, assume that $C(\mathcal{P})$ is $(n-1)$-connected and
	 that $\forall~S\in\mathcal{P}\setminus\mathcal{H}$ with $h(S)=s\leq n$ we have that $\mathrm{Del}^S$ is $(s-2)$-connected (or empty if $s=0$). Then
	 \begin{itemize}
	 \item[i)] if $\abs{\mathcal{J}_S}$ is $(n-1-s)$-connected, $C(\mathcal{H})$ is $(n-1)$-connected,
	 
	 \item[ii)] if $\abs{\mathcal{J}_S}$ is $(n-1-s)$-acyclic, $C(\mathcal{H})$ is $(n-1)$-acyclic.
	\end{itemize}
\end{lemma}

When considering the case of Artin groups we will apply the previous lemma with $\mathcal{P}$ the \textbf{spherical  poset}, i.e.:
$$\mathcal{P}=\lbrace A_\Delta\mid\Delta\subset\Gamma, A_\Delta\text{ is spherical}\rbrace.$$
Note that $A_\Delta$ spherical implies that $\Delta\subset\Gamma$ is a clique but, unless $A_\Gamma$ is of FC-type, there could be cliques $\Delta\subset\Gamma$ for which $A_\Delta$ is not spherical. 
Indeed, $C(\mathcal{P})$ is the  \textbf{Deligne complex} (sometimes called \textbf{modified Deligne complex}) of the Artin group, according to the classical definition of Charney-Davis in \cite{Charney-Davis}. Moreover, by \cite{Charney-Davis}, $C(\mathcal{P})$ is homotopy equivalent to the universal cover of the Salvetti complex, which is contractible if the Artin group $A_\Gamma$ satisfies the $K(\pi,1)$-conjecture.\\
Observe that $\mathcal{P}$ is subgroup closed which implies that for any $S\in\mathcal{P}$,
$\lbrace T\in\mathcal{P}\mid T\leq S\rbrace=\lbrace S\cap T\mid T\in\mathcal{P}\rbrace$.  
Another technical lemma we need is the following well known consequence of a theorem by Deligne:
\begin{lemma}\label{Deligne}
	Let  $S=A_{\Delta}$ be a spherical Artin group over the complete graph $\Delta$ with $0\neq s=\abs{\Delta}$. Let $\mathrm{Del}^S$ be the geometric realization of the coset poset of the
special proper subgroups of $A_{\Delta}$, i.e of the subgroups generated by possibly empty subsets $V\subsetneq\Delta$.	
	Then $\mathrm{Del}^S$ is homotopically equivalent to a wedge of $(s-1)$-spheres, so it is $(s-2)$-acyclic and $(s-2)$-connected.
\end{lemma}
\begin{proof} For each  $v\in\Delta$, put $\Delta_v=\Delta\setminus \{v\}$. Given $g\in S$, let
$gX_{v}$ be the subcomplex of $\mathrm{Del}^S$ induced by $\{gA_{\Omega}\mid \Omega\subseteq\Delta_v\}.$  As
each special proper subgroup of $A_\Delta$ is inside one of the form $A_{\Delta_v}$, the set $\{gX_{v}\mid g\in S,v\in\Delta\}$ is a covering of $\mathrm{Del}^S$.  

Observe that for any $\Delta_1\subset\Delta$, if $g_1,g_2\in S$ are such that $g_1A_{\Delta_1}\neq g_2A_{\Delta_1}$, then 
$g_1A_{\Delta_1}\cap g_2A_{\Delta_1}=\emptyset$. And if we have $\Delta_1,\Delta_2$ and $g_1,g_2\in S$ so that $g_1A_{\Delta_1}\cap g_2A_{\Delta_2}\neq\emptyset$, then for some $x\in A_{\Delta_1}$, $y\in A_{\Delta_2}$, $g_1x=g_2y$ so 
$$g_1A_{\Delta_1}\cap g_2A_{\Delta_2}=g_1(A_{\Delta_1}\cap xA_{\Delta_2})=g_1xA_{\Delta_1\cap\Delta_2}$$
because if $g\in A_{\Delta_1}\cap xA_{\Delta_2}$, then $x^{-1}g\in A_{\Delta_1}\cap A_{\Delta_2}=A_{\Delta_1\cap\Delta_2}$ (cf. \cite{VanderLek}). This implies that any set of pairwise distinct $g_0X_{v_0},\ldots,g_kX_{v_k}$ has empty intersection if two of the $v_i$'s are equal and also that whenever such a set has non-empty intersection, then the intersection is the subcomplex induced by 
$$\{gA_{\Omega}\mid \Omega\subseteq\Delta_{v_1}\cap\ldots\cap\Delta_{v_k},g\in g_0A_{\Delta_{v_0}\cap\ldots\cap\Delta_{v_k}}\}$$ which is contractible (it is the geometric realization of a poset with a maximal element). 
The last property implies that $\mathrm{Del}^S$ is homotopy equivalent to the nerve of the covering (cf. \cite{BrownBook} Chapter VII Theorem 4.4). 
The nerve has as $k$-simplices the sets  $\{g_0X_{v_0},\ldots,g_kX_{v_k}\}$ with non empty intersection. The discussion above implies that the cardinality of such a set can be at most $s$ so this nerve  is a complex of dimension  $s-1$. 
This is precisely the complex considered in \cite{Deligne} Section 4.18 which is shown to be a wedge of $(s-1)$-spheres (cf. \cite{Deligne} Theorem 2.15).
\end{proof}

Therefore, to be able to apply Lemma \ref{Lemma 2.9} for the case of an Artin group $A_\Gamma$ satisfying the $K(\pi,1)$-conjecture, we would need to find a suitable $\mathcal{H}$ such that $\abs{\mathcal{J}_S}$ is $(n-1-s)$-acyclic $\forall~S\in\mathcal{P}\setminus\mathcal{H}$ with $h(S)=s$. It will be seen later that we can take 
$\mathcal{H}=\{S=A_\Delta\in\mathcal{P}\mid\chi(Z(A_\Delta))=0\}$. Finally, note that when dealing with Artin groups, posets of spherical subgroups can be identified with posets of {\bf spherical subgraphs}, i.e., subgraphs $\Delta\subseteq\Gamma$ such that $A_\Delta$ is spherical. For example, if $S=A_\Delta$ we can set: 
$$\mathcal{J}_\Delta=\{X\subseteq\Gamma\mid A_X\in\mathcal{J}_S\}.$$

\section{Centre of Spherical Irreducible Artin Groups}\label{Section 3}
One natural way in which one can deduce that a given character lies in the $\Sigma$-invariant of a group with non trivial centre is via Lemma \ref{Lemma 2.4}. The centre of irreducible Artin groups is well known and we will use that to determine big parts of their $\Sigma$-invariants. This strategy does not help in the non spherical case, as Jankiewicz-Schreve proved in \cite{Jankiewicz-Schreve} that every Artin group satisfying the $K(\pi,1)$-conjecture without a spherical factor has trivial centre.
\begin{definition}\label{def:lift}
	Let $A_\Gamma$ be an Artin group and $q:A_\Gamma\to W_\Gamma$ the natural quotient homomorphism. One can define a set-theoretic section of $q$ in the following way: let $w\in W_\Gamma$ and $v_1\cdots v_n$ a word of minimal length representing $w$ in terms of the generators. The section is defined as $w\to a_w=v_1\cdots v_n$, where in this case we are seeing $v_i$ as the generators of $A_\Gamma$. It follows from Tits' solution to the word problem for Coxeter groups that this section is well defined (cf. \cite{Charney-Davis}). We will say that $a_w$ is the {\bf lift} of $w$ in $A_\Gamma$.
\end{definition}
The centre of a spherical Artin group can be described in terms of its Garside element, as stated in the next result.
\begin{proposition}[\cite{Dehorny} B, IX, Corollary 1.39]\label{Proposition 3.1}
	Let $A_\Gamma$ be a spherical Artin group and $W_\Gamma$ its associated Coxeter group. Let $z_\Gamma$ be the lift of the longest element of $W_\Gamma$ in $A_\Gamma$. Then $Z(A_\Gamma)$ is infinite cyclic generated by $z_\Gamma^n$ where $n$ is $1$ or $2$ depending on $\Gamma$. The element $z_\Gamma$ is called the \textbf{Garside element} of $A_\Gamma$.
\end{proposition}
This result makes use of the fact that  a Coxeter group has a unique longest element if (and only if) it is finite (\cite{Humphreys} Section 1.8). However, although the longest element is unique, it may have many different reduced expressions.\\

 The first step to understand the $\Sigma$-invariants of a group is to determine its character sphere. This is easy for spherical Artin groups:
\begin{lemma}\label{spherespherical}
	\begin{enumerate}
		\item If $G\in\lbrace A_{\mathbb{A}_n},A_{\mathbb{D}_n},A_{\mathbb{E}_6},A_{\mathbb{E}_7},A_{\mathbb{E}_8},A_{\mathbb{H}_3},A_{\mathbb{H}_4},A_{\mathbb{I}_2(k)}\text{ for }k\text{ odd}\rbrace$
		then  $G$ has cyclic abelianization and $S(G)=\mathbb{S}^0$.
		\item If $G\in\lbrace A_{\mathbb{B}_m},A_{\mathbb{F}_4},A_{\mathbb{I}_2(2k)}\rbrace$ then  the abelianization of $G$ has rank 2 and $S(G)=\mathbb{S}^1$.
	\end{enumerate}
\end{lemma}

In the cases where $S(G)=\mathbb{S}^0$ the $\Sigma$-invariants are easy to describe.
\begin{corollary}
	If $G$ is a spherical irreducible Artin group such that $S(G)=\mathbb{S}^0$ and $R$ is a commutative ring then $\Sigma^n(G)=\Sigma^n(G,R)=S(G)~\forall~n\geq 0$.
\end{corollary}	\begin{proof}
		Let $w$ be the longest word of the associated finite Coxeter group of $G$. Since the generators are involutions in the Coxeter group, $w$ can be written as a product of generators with no powers. We have $S(G)=\lbrace[\chi],[-\chi]\rbrace$, where $\chi:G\to\mathbb{R}$ maps each generator of $G$ to $1$. Hence $\chi(w)\neq0$ and, by Proposition \ref{Proposition 3.1}, $w^2\in Z(G)$. Therefore $\chi(Z(G))\neq0$ and, by Lemmas \ref{Lemma 2.4} and \ref{Lemma 2.7}, we have $[\chi],[-\chi]\in\Sigma^n(G)~\forall~n>0$.
	\end{proof}

For the cases when $S(G)=\mathbb{S}^1$ we need a stronger theorem, which allows us to compute explicitly the longest word.
\begin{theorem}[\cite{Bourbaki} Chapter 5, Section 6.2, Theorem 1]\label{Theorem 3.5}
	Let $W$ be an irreducible finite Coxeter group, $w_0\in W$ the longest element of $W$, $N=l(w_0)$ the length of $w_0$,  $\lbrace s_1,\dots,s_n\rbrace$ the set of standard generators of $W$, $w=s_1\dots s_n$ and $h=\text{ord}(w)$. Then:
	\begin{enumerate}
		\item $2N=nh$
		\item If $h$ is even $w_0=w^{h/2}$.
	\end{enumerate}
\end{theorem}
The length of the longest word of $W$ coincides with the number of positive roots (cf. \cite{Humphreys} Section 1.8). A computation of the positive roots can be found in \cite{Bourbaki} Chapter 6, Section 4. In particular $\mathbb{B}_n$ has $n^2$ positive roots, $\mathbb{F}_4$ has $24$ positive roots and $\mathbb{I}_{2}(k)$ has $k$ positive roots.
\begin{examples}\label{Examples 3.5}
    \begin{enumerate}
        \item With the notation of Theorem \ref{Theorem 3.5} consider the group $\mathbb{B}_n$, where the generators are labelled as follows (we use here  Dynkin diagrams):
		$$
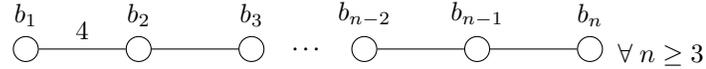
\begin{tikzpicture}[main/.style = {draw, circle},node distance={15mm}] 
			\node[label={$b_1$}][main] (1) {};
			\node[label={$b_2$}][main] (2) [right of=1] {}; 
			\node[label={$b_3$}][main] (3) [right of=2] {}; 
			\draw[-] (1) -- node[above] {$4$}   (2);
			\draw[-] (2) -- (3);
			\node[label={$b_{n-2}$}][main] (4) [right of=3] {};
			\node[label={$b_{n-1}$}][main] (5) [right of=4] {}; 
			\node[label={$b_n$}][main] (6) [right of=5] {}; 
			\draw[-] (4) --  (5);
			\draw[-] (5) -- (6);
			\node at ($(3)!.5!(4)$) {\ldots};
		\end{tikzpicture}~\forall~n\geq3$$
		$h=\frac{2N}{n}=\frac{2n^2}{n}=2n$ and, since $h$ is even, $w_0=(b_1\dots b_n)^n$ is the longest word of the Coxeter group. Hence, if $G=A_{\mathbb{B}_n}$ then $w_0^2\in Z(G)$. In $G/G'$ we have $\overline{b_2}=\dots=\overline{b_n}$, so for every character $\chi:G\to\mathbb{R}$: $$\chi(w_0)=n\chi(b_1~b_2^{n-1})=n(\chi(b_1)+(n-1)\chi(b_2))$$
		Then, if $\chi(b_1)+(n-1)\chi(b_2)\neq 0$, Lemma \ref{Lemma 2.4} tells us that $[\chi]\in\Sigma^n(G)~\forall~n\geq0$. This implies the following:
		$$S(G)\setminus\lbrace\pm[(n-1,-1)]\rbrace\subseteq\Sigma^n(G)~\forall~n\geq 0$$
		where $(n-1,-1)$ is the character $\chi$ given by $\chi(b_1)=n-1$, $\chi(b_2)=-1$.
		\item In the same way, if $G=A_{\mathbb{F}_4}$ or $G=A_{\mathbb{I}_2(2k)}$ we can prove that $S(G)\setminus\lbrace\pm[(1,-1)]\rbrace\subseteq\Sigma^n(G)~\forall~n\geq 0$. For the $G=A_{\mathbb{I}_2(2k)}$ case, the centre is generated by $(ab)^k$ where $a$ and $b$ are the standard generators. And for $\mathbb{F}_4$, which corresponds to the following Dynkin diagram:
		$$
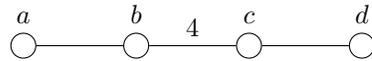
\begin{tikzpicture}[main/.style = {draw, circle},node distance={15mm}] 
		\node[label={$a$}][main] (1) {};
		\node[label={$b$}][main] (2) [right of=1] {}; 
		\node[label={$c$}][main] (3) [right of=2] {}; 
		\node[label={$d$}][main] (4) [right of=3] {}; 
		\draw[-] (1) --  (2);
		\draw[-] (2) -- node[above] {$4$}  (3);
		\draw[-] (3) --  (4);
		\end{tikzpicture}$$
	we have that the centre is generated by $(abcd)^6$.
	\end{enumerate}
\end{examples}
Using this we can easily determine the invariants of ${\mathbb{I}_2(k)}$ (this is in \cite{Meier-Meiner-VanWyk 2} but we include a short proof).
\begin{lemma}\label{Lemma 3.6}
	If $G=A_{\mathbb{I}_2(2k)}$, then $\Sigma^n(G)=\Sigma^n(G,\mathbb{Z})=S(G)\setminus\lbrace\pm[(1,-1)]\rbrace~\forall~n\geq 1$.
\end{lemma}	\begin{proof}
	    Use Example \ref{Examples 3.5}.2 to prove $\supseteq$ and Brown's criteria for one-relator groups (cf. \cite{Brown}) to deduce that $\pm[(1,-1)]\notin\Sigma^1(G)$.
	\end{proof}

The cases of $\mathbb{F}_4$ and $\mathbb{B}_n$ will be tackled later.

\section{Strong $n$-link condition}\label{Strong}
In this section, we will define the strong $n$-link condition, which is the combinatorial condition needed in the statement of Theorem \ref{Theorem 1.2}. To do so, we will recall some results that Theorem \ref{Theorem 1.2} generalizes and give an overview of the previous versions of this condition. This will motivate our definition.\\ 
In the case when $G=A_\Gamma$ is a RAAG, Meier-VanWyk gave a description of $\Sigma^1$ in \cite{Meier-VanWyk} in terms of a combinatorial property of the graph, which was generalized by Meier-Meinert-VanWyk to give a complete computation of the $\Sigma$-invariants in \cite{Meier-Meiner-VanWyk}. This was proved by considering  the action on the Deligne complex associated to a RAAG. This idea was used later by the same authors in \cite{Meier-Meiner-VanWyk 2} to give a sufficient condition for certain particular characters $\chi:A_\Gamma\to\mathbb{Z}$ to belong to $\Sigma^n(A_\Gamma,\mathbb{Z})$ for Artin groups of FC-type. In \cite{Blasco-Cogolludo-Martinez} Blasco-Cogolludo-Martinez generalized this to more general characters in the case of even Artin groups of FC-type. To state these theorems we need to introduce some previous   definitions.\\
If $\Gamma_1\subset\Gamma$ is a subgraph and $v\in \Gamma$, then the \textbf{link} $\text{lk}_{\Gamma_1}(v)$ is the subgraph induced by the set of vertices of $\Gamma_1$ that are adjacent to $v$ in $\Gamma$. Moreover, this definition may be extended for subsets $\Delta\subset\Gamma$ by setting $\text{lk}_{\Gamma_1}(\Delta)=\bigcap\limits_{v\in V(\Delta)}\text{lk}_{\Gamma_1}(v)$.\\
Now, let $\chi:A_\Gamma\to\mathbb{R}$ and $n\geq0$ an integer. Meier-Meinert-VanWyk defined in \cite{Meier-Meiner-VanWyk} the \textbf{living subgraph} $\mathrm{Liv}^\chi$ as the subgraph of $\Gamma$ induced by the vertices $v$ such that $\chi(v)\neq 0$. The vertices of $\mathrm{Liv}^\chi$ are called \textbf{living vertices} while the vertices not in $\mathrm{Liv}^\chi$ are called \textbf{dead vertices}. In the case when $A_\Gamma$ is a RAAG, we will say that $\chi$ satisfies the \textbf{$n$-link condition} if $\reallywidehat{\text{lk}_{\mathrm{Liv}^\chi}(\Delta)}$ is $(n-1-\abs{\Delta})$-acyclic for any spherical $\Delta\subset\Gamma\setminus\mathrm{Liv}^\chi$ with $\abs{\Delta}\leq n$. Here  $\reallywidehat{\text{lk}_{\mathrm{Liv}^\chi}(\Delta)}$ is the \textbf{flag complex} associated to the graph $\text{lk}_{\mathrm{Liv}^\chi}(\Delta)$, i.e. the simplicial complex obtained by adding cells of the corresponding dimension to every spherical subset (cliques are the spherical subsets in the RAAG case). 
The main theorem from \cite{Meier-Meiner-VanWyk} is that if $G$ is a RAAG, $[\chi]\in\Sigma^n(G,\mathbb{Z})$ if and only if $\chi$ satisfies the $n$-link condition. They also proved a homotopical version i.e. that $[\chi]\in\Sigma^n(G)$ if and only if $\chi$ satisfies the \textbf{homotopical $n$-link condition}, which is obtained substituting $(n-1-\abs{\Delta})$-acyclic with $(n-1-\abs{\Delta})$-connected in the definition above.

For Artin groups, the notion of  \textbf{living subgraph} was generalized in \cite{Meier} as follows. An edge $e=\lbrace u,v\rbrace\in \Gamma$ is said to be \textbf{dead} if $l(e)>2$ and $\chi(u)+\chi(v)=0$. The \textbf{living subgraph} is the subgraph $\mathrm{Liv}^\chi$ of $\Gamma$ obtained after removing all dead vertices and the interior of all dead edges. It is clear that this subgraph coincides with the one defined above for the RAAG case. We will also define $\mathrm{Liv}_0^\chi$ to be the subgraph of $\Gamma$ induced by the living vertices.

In this context, Meier proved in \cite{Meier} that every non-zero character $\chi$ with $\mathrm{Liv}^\chi$ connected and dominant ($\mathrm{Liv}^\chi$ is dominant if every vertex of $\Gamma\setminus \mathrm{Liv}^\chi$ is connected to some vertex of $\mathrm{Liv}^\chi$) belongs to $\Sigma^1(A_\Gamma)$. He also proved that if $[\chi]\in\Sigma^1(A_\Gamma)$ then  $\mathrm{Liv}_0^\chi$ is connected and dominant, and so $\mathrm{Liv}^\chi$ is dominant.
Whether the converse is true is a conjecture (cf. \cite{Almeida}), i.e. whether if $[\chi]\in\Sigma^1(A_\Gamma)$ then $\mathrm{Liv}^\chi$ is connected. Already Meier observed that using the fact that the invariants are open and the discrete characters are dense it suffices to prove the conjecture for discrete characters. Some particular cases of the conjecture are solved: for example, if $\Gamma$ is a complete graph and all the edges are labelled with the same number \cite{Meier}, if $\Gamma$ is connected and $\pi_1(\Gamma)$ is either trivial or free of rank $1$ or $2$ \cite{Meier-Meiner-VanWyk 2},\cite{Almeida-Kochloukova},\cite{Almeida}, if $\Gamma$ is in certain family of complete graphs with $4$ edges \cite{Almeida-Kochloukova 2}, if $\Gamma$ is in certain family of minimal graphs with $\pi_1(\Gamma)$ free of arbitrary rank \cite{Almeida 2} and if $\Gamma$ is even and has the property that if there is a closed reduced path in $\Gamma$ with all labels bigger than 2, then the length of such path is always odd \cite{Kochloukova}.

In \cite{Meier-Meiner-VanWyk 2} the description of the Sigma invariants for RAAGs was partially extended to prove that if $G=A_\Gamma$ is of FC-type,  $[\chi]\in\Sigma^n(G,\mathbb{Z})$ if the flag complex $\widehat{\Gamma}$ is $(n-1)$-acyclic and $\chi(Z(A_\Delta))\neq0$ for any non empty $\Delta\subset\Gamma$. 
Moreover, with the extra hypothesis of $\Gamma$ being even, Blasco-Cogolludo-Martinez defined in \cite{Blasco-Cogolludo-Martinez} a generalization of the $n$-link condition as follows.
Consider  the set $\mathcal{B}^\chi$ of subgraphs $\Delta\subseteq\Gamma$ with $A_\Delta$ spherical such that each $v\in\Delta$ is either dead or belongs to a dead edge in $\Delta$. We say that $\chi$ satisfies the \textbf{strong $n$-link condition} if $\reallywidehat{\text{lk}_{\mathrm{Liv}^\chi}(\Delta)}$ is $(n-1-\abs{\Delta})$-acyclic for any $\Delta\in\mathcal{B}^\chi$ with $\abs{\Delta}\leq n$. Analogously, we say that $\chi$ satisfies the \textbf{strong homotopical $n$-link condition} in the same way by substituting $(n-1-\abs{\Delta})$-acyclic with $(n-1-\abs{\Delta})$-connected. The main theorem from \cite{Blasco-Cogolludo-Martinez} is that for an even Artin group of FC-type a character $\chi:G\to\mathbb{R}$ satisfying the strong $n$-link condition belongs to $\Sigma^n(G,\mathbb{Z})$. As in the RAAG case, the homotopical version is proved analogously.\\
It was shown in \cite{Blasco-Cogolludo-Martinez} that $\mathcal{B}^\chi$ is precisely the poset consisting of those spherical $\Delta\subset\Gamma$ such that $\chi(Z(A_\Delta))= 0$, which is why this theorem generalizes the result in \cite{Meier-Meiner-VanWyk 2}. 
To generalize this theorem to the case of an arbitrary Artin group $A_\Gamma$ satisfying the $K(\pi,1)$-conjecture we have to extend the notion of the strong $n$-link condition.  Let us begin by fixing some notation:

\begin{definition}
	Let $G=A_\Gamma$ be an Artin group and $\chi:A_\Gamma\to\mathbb{R}$ a character. If $\chi(v)=0$ we say that the vertex is \textbf{dead} and we say that it is a \textbf{living} vertex otherwise. An edge $e=\lbrace v,w\rbrace$ with label $l(e)$ is said to be
	\begin{enumerate}
	\item \textbf{$2$-dead} if $l(e)\geq 4$ is even and $\chi(v)+\chi(w)=0$,
	\item \textbf{$n$-dead} for $n>2$ if $l(e)=4$, $e$ belongs to a type $\mathbb{B}_n$ subgraph and $(n-1)\chi(v)+\chi(w)=0$, where $w$ is the vertex that connects $e$ with an edge labelled with a $3$.
	\end{enumerate} 
	We also set
	$$\mathcal{B}^\chi=\lbrace \Delta\subseteq\Gamma\,\mathrm{ spherical}\mid\chi(Z(A_\Delta))= 0\rbrace$$
	\end{definition}
	By Lemma \ref{Lemma 2.4}, $\mathcal{B}^\chi\supseteq\lbrace \Delta\subseteq\Gamma\,\mathrm{ spherical}\mid[\restr{\chi}{A_\Delta}]\not\in\Sigma^\infty(A_\Delta,\mathbb{Z})\rbrace$. It will be shown in the next Lemma that indeed we have an equality.
	\begin{lemma}\label{Corollary 5.7}
	If $A_\Gamma$ is an Artin group  and $\chi:A_\Gamma\to\mathbb{R}$ is a character then $ \mathcal{B}^\chi$ is
	the set of those $\Delta\subseteq\Gamma$ spherical such that if $A_\Delta=A_{\Delta_1}\times\ldots\times A_{\Delta_t}$ is the decomposition into irreducible components then for any $\Delta_i$ we have
	$$ \begin{cases}
	\text{either all vertices of $\Delta_i$ are dead }\\
	\text{or }\Delta_i=A_{\mathbb{I}_2(2k)} \text{ where the edge is $2$-dead}\\
	\text{or }\Delta_i=A_{\mathbb{F}_4}\text{ where the edge labelled with a $4$ is $2$-dead}\\
	\text{or }\Delta_i=A_{\mathbb{B}_n}\text{ where the edge labelled with a $4$ is $n$-dead}\\
\end{cases}$$
	\end{lemma}	\begin{proof}
	    Apply  the previous computation of the centers of spherical Artin groups and the fact that if $A_\Delta$ decomposes as $A_{\Delta_1}\times\dots\times A_{\Delta_t}$ then $\Delta\in\mathcal{B}^\chi\Leftrightarrow\Delta_i\in\mathcal{B}^\chi~\forall~i=1,\dots,n$.
	\end{proof}

Next, we define a simplicial complex that will be crucial for our version of the strong $n$-link condition.

\begin{definition}\label{Definition 6.4} Let $\Gamma$ be a finite simple graph, $\chi:A_\Gamma\to\mathbb{R}$ a character and  $\Gamma_1,\Delta\subset\Gamma$ with $\Delta$ a spherical subgraph, i.e., a subgraph such that $A_\Delta$ is spherical. 
	Consider the (non flag) simplicial complex $\slk_{\Gamma_1,\Delta}^\chi$ 
	having an $m$-cell, $m\geq 1$, for each $m+1$ spherical  subgraph  $X\subseteq\lk_{\Gamma_1}(\Delta)$ such that $X\cup\Delta$ is also spherical (such an $X$ is necessarily a clique). We call this complex the {\bf spherical link complex}. Then,  we define a new simplicial complex $L_\Delta^\chi$ as follows. Consider $\slk_{\Gamma,\Delta}^\chi$ and  delete: 	\begin{enumerate}
	\item dead vertices
	\item open edges with even label $\geq4$ which are $2$-dead,
		\item open cells $\sigma\subset\slk_\Delta^\chi$ that correspond to a copy of $\mathbb{B}_n$ with $\chi(Z(A_{\mathbb{B}_n}))=0$, i.e. such that the edge $e\in \sigma$ with label 4 is $n$-dead,		
		\item open edges $e$ with label $2$ such that $\Delta\cup e$ contains a new irreducible component $\mathbb{F}_4$ and the edge of $\mathbb{F}_4$ with label 4 is $2$-dead.	
	\end{enumerate}
	(Observe that deleting an open cell in a simplicial complex means that all the cells containing it are also deleted). We get a new simplicial complex that we denote $L_\Delta^\chi$.
\end{definition}

The strong $n$-link condition is defined as follows:
\begin{definition}
	With the same notation as before, assume that for any $\Delta\in\mathcal{B}^\chi$ with $\abs{\Delta}\leq n$ the simplicial complex $L_\Delta^\chi$ is $(n-1-\abs{\Delta})$-acyclic. Then, $\chi$ is said to satisfy the \textbf{strong homological $n$-link condition}. If we substitute the hypothesis of $(n-1-\abs{\Delta})$-acyclic with $(n-1-\abs{\Delta})$-connected we say that $\chi$ satisfy the \textbf{strong homotopical $n$-link condition}.
\end{definition}
   
\begin{examples}\label{strongF4Bm} We use the same notation as in Example \ref{Examples 3.5}.
       \begin{enumerate} 
        \item Consider the Artin group $A_{\mathbb{F}_4}$ and the characters $\chi=\pm(1,-1)$. Then, $\mathcal{B}^\chi=\lbrace\emptyset,(b,c),\mathbb{F}_4\rbrace$,
        $L^\chi_\emptyset$ consists of the two triangles $abd$ and $cbd$  which are joined via an edge, $L^\chi_{(b,c)}=\{a,d\}$ and $L_{\mathbb{F}_4}^\chi=\emptyset$. Thus, $L^\chi_\emptyset$ is contractible and $L^\chi_{(b,c)}$ is $(-1)$-connected. Therefore, $\chi$ satisfies the strong homotopical $2$-link condition but not the strong homological $3$-link condition.
		\item	Consider the Artin group $A_{\mathbb{B}_n}$ and the characters $\chi=\pm(n-1,-1)$. Then, $\mathcal{B}^\chi=\lbrace\emptyset,\mathbb{B}_n\rbrace$, $L_\emptyset^\chi$ is the full simplicial set with $n$ vertices with the top $n$-cell deleted and $L^\chi_{\mathbb{B}_n}=\emptyset$. Thus $L_\emptyset^\chi\simeq\mathbb{S}^{n-2}$ is $(n-3)$-connected and $\chi$ satisfies the strong homotopical $(n-2)$-link condition but not the strong homological $(n-1)$-link condition.
\end{enumerate}
\end{examples}

\section{$\Sigma$-Invariants of Artin groups satisfying the $K(\pi,1)$-conjecture.}\label{Section 5}

In this section, we will prove Theorem \ref{Theorem 1.2} by following the ideas introduced in section \ref{Section 2.2}.\\
Let us fix some notation.  Let $A_\Gamma$ be an Artin group satisfying the $K(\pi,1)$-conjecture, $\chi:A_\Gamma\to\mathbb{R}$ a character and  $\Delta\in\mathcal{B}^\chi$ ($\Delta$ could be the empty subset, which always lies in $\mathcal{B}^\chi$).
Recall that in Section \ref{Section 2.2} we wanted to find a suitable $\mathcal{H}\subset\mathcal{P}$ such that $\abs{\mathcal{J}_S}$ is $(n-1-s)$-acyclic $\forall~S\in\mathcal{P}\setminus\mathcal{H}$ with $h(S)=s$. Indeed, we will take $\mathcal{H}=\mathcal{P}\setminus\{A_\Delta\mid\Delta\in\mathcal{B}^\chi\}$. Observe that in this case:
$$\mathcal{J}_S=\lbrace T\in\mathcal{H}\mid S\subset T\rbrace=\lbrace T\in\mathcal{P}\mid S\subset T\text{ and }\chi(Z(T))\neq 0\rbrace.$$
\begin{definition} Let $X\subseteq\Gamma$ be spherical, we set
$$\mathcal{B}^\chi(X)=\max\lbrace Y \subset X\mid Y\in\mathcal{B}^\chi\rbrace=\max\lbrace Y \subset X\mid\chi(Z(A_Y))=0\rbrace.$$ 
 %which computes the biggest subgraph of $X$ such that its corresponding subgroup lies in $\mathcal{B}^\chi$.
\end{definition}
To see that this is well defined, consider first the case when $X$ is irreducible. If there is some $v\in X$ with $\chi(v)=0$, then $\mathcal{B}^\chi(X)$ is precisely the subgraph of $X$ induced by those $v\in X$ with $\chi(v)=0$. So we assume $\chi(v)\neq 0$ for all $x\in X$. Then, if $X\not\in\lbrace\mathbb{B}_m,\mathbb{F}_4,\mathbb{I}_2(2k)\rbrace$, equivalently, if $A_X$ has cyclic abelianization, we have 
$\mathcal{B}^\chi(X)=\{\emptyset\}$. If $X\in\lbrace\mathbb{B}_m,\mathbb{I}_2(2k)\rbrace$, then there is at most one $\emptyset\neq Y \subset X$ with $\chi(Z(A_Y))=0$ so $\mathcal{B}^\chi(X)$ is well defined. And if $X=\mathbb{F}_4$, then either $\mathcal{B}^\chi(X)=\{\emptyset\}$ or $\mathcal{B}^\chi(X)=\mathbb{F}_4$.
In the non irreducible case,  set $X=X_1\star\dots\star X_{k}$ with each $X_i$ irreducible where $\star$ denotes the join. Then we have
$$\mathcal{B}^\chi(X)=\mathcal{B}^\chi(X_1)\star\dots\star\mathcal{B}^\chi(X_{k}).$$
The previous  discussion also implies the following easy properties of $\mathcal{B}^\chi(X)$.
\begin{lemma}\label{Lemma 6.2}
	\begin{itemize}
		\item[i)] If $X_1\subset X_2\subset \Gamma$ are spherical then $\mathcal{B}^\chi(X_1)\subseteq\mathcal{B}^\chi(X_2)$.
		\item[ii)] If $X\subset\Gamma$ is spherical and $\Delta=\mathcal{B}^\chi(X)$, then $\Delta=\mathcal{B}^\chi(X_v)~\forall~v\in X\setminus \Delta$, where $X_v=X\setminus\lbrace v\rbrace$.
	\end{itemize}
	\end{lemma}

We will consider also the following  posets associated to  a given $\Delta\in\mathcal{B}^\chi$:
$$\mathcal{L}_\Delta^\chi=\text{poset of simplices in }L_\Delta^\chi,$$
$$\mathcal{T}_\Delta^\chi=\lbrace \Delta\subsetneq X\subseteq\Gamma\, \text{ spherical}\mid\mathcal{B}^\chi(X)=\Delta\rbrace,$$
$$\mathcal{J}_\Delta=\lbrace\Delta\subseteq X\subseteq\Gamma\, \text{ spherical}\mid\chi(Z(A_X))\neq 0\rbrace$$
Setting $S=A_\Delta$, $\mathcal{J}_\Delta$ can be identified with the poset $\mathcal{J}_S$ considered above, and in particular both have the same geometric realization.

\begin{proposition}\label{Proposition 6.3} Let $\Delta\in\mathcal{B}^\chi$, i.e., $\Delta\subseteq\Gamma$ spherical with $\chi(Z(A_\Delta))=0$. 
The following poset maps
$$\tau:\mathcal{L}_\Delta^\chi\to\mathcal{T}_\Delta^\chi, \quad Y\mapsto Y\cup\Delta,$$
$$\iota:\mathcal{T}_\Delta^\chi\hookrightarrow \mathcal{J}_\Delta\text{ the inclusion}$$
are well defined and induce homotopy equivalences between the corresponding geometric realizations.
\end{proposition}
\begin{proof} Consider first $\tau$. To see that it is well defined, recall that for any $Y\in\mathcal{L}_\Delta^\chi$, $Y$ is a spherical subset of $\Gamma$ such that $Y\cup\Delta$ is also spherical. Moreover, as $\Delta\subseteq Y\cup\Delta$, Lemma \ref{Lemma 6.2} i) tells us that  $\Delta\subseteq\mathcal{B}^\chi(Y\cup\Delta)$. If we had $\Delta\subsetneq\mathcal{B}^\chi(Y\cup\Delta)$, then there would be some irreducible factor $A_{\hat\Delta}$ of $A_{Y\cup\Delta}$ with $\hat\Delta\subseteq\mathcal{B}^\chi(Y\cup\Delta)$, i.e. with $\chi(Z(A_{\hat\Delta}))=0$ so that either $A_{\hat\Delta}$ contains properly some irreducible factor of $A_\Delta$ or with $\hat\Delta$ in the spherical link of $\Delta$. The second case is impossible by the definition of $L_\Delta^\chi$. For the first, note that by definition of $L_\Delta^\chi$, $\chi(v)\neq 0$ for any $v\in Y$ so the only case in which an irreducible factor such that the character vanishes in the centre can embed properly in another one in the same conditions is for $A_{\mathbb{I}_2(4)}\subset A_{\mathbb{F}_4}$. In this case we would have an edge with label 2 in $L_\Delta^\chi$ of a kind forbidden by Definition \ref{Definition 6.4}.

To see that it is a homotopy equivalence, consider $X\in\mathcal{T}_\Delta^\chi$. Let $Y=X\setminus\Delta$. Decomposing $A_X$ in irreducible components and using that $\mathcal{B}^\chi(X)=\Delta$ we see that $Y$ lies in the spherical link of $\Delta$, no vertex of $Y$ can have zero $\chi$-value and indeed $Y$ must be a spherical subset of $L_\Delta^\chi$ so it corresponds to a simplex in $\mathcal{L}_\Delta^\chi$. This means that $\tau$ is surjective thus the set
$$\tau^{-1}_\leq(X)=\{\sigma\in\mathcal{L}_\Delta^\chi\mid Y\cup\Delta\subseteq X\}$$
has a maximal element, so it has a contractible geometric realization which implies that $\tau$ induces a homotopy equivalence by Quillen's Lemma (cf. \cite{Benson} Lemma 6.5.2). 

To see that  also $\iota$ induces a homotopy equivalence, fix $Y\in \mathcal{J}_\Delta$ and consider the set
$$\iota^{-1}_\leq(Y)=\{X\in\mathcal{T}_\Delta^\chi\mid X\subseteq Y\}.$$
As $\Delta\subseteq Y$, Lemma \ref{Lemma 6.2} i) tells us that $\Delta\subseteq\mathcal{B}^\chi(Y)$ and both are spherical. Consider the decomposition as a product of irreducible components of the associated Artin groups. The $A_{\mathbb{B}_n}$ or $A_{\mathbb{F}_4}$ components of $A_\Delta$ remain in the Artin group associated to $\mathcal{B}^\chi(Y)$, there could be some new components, and also there could be some $A_{\mathbb{I}_2(2k)}$ components that remain and some $A_{\mathbb{I}_2(4)}$ components that embed into new $A_{\mathbb{F}_4}$ components in $\mathcal{B}^\chi(Y)$. Therefore $\mathcal{B}^\chi(Y)=\Delta\cup T_1\cup T_2$ where $T_1$ is a spherical set consisting of vertices with zero $\chi$-value only and $A_{T_2}$ is a direct product of groups generated by  pair of vertices forming an edge with label 2 that complement some  $A_{\mathbb{I}_2(4)}$ in $A_\Delta$ to produce an $A_{\mathbb{F}_4}$ in $A_{\mathcal{B}^\chi(Y)}$. Put $T_2=L_1\cup\ldots\cup L_r$ for $L_i$ those edges. Besides, as $\chi(Z(A_Y))\neq 0$, $Y=\mathcal{B}^\chi(Y)\cup Y_1$ for some $Y_1\neq\emptyset$. Now, any $X\in\mathcal{T}_\Delta^\chi$ such that $X\subseteq Y=\Delta\cup T_1\cup T_2\cup Y_1$ must be of the form $\Delta\cup Z$ for some $Z\subseteq Y\setminus\Delta$ and the condition that $\mathcal{B}^\chi(X)=\Delta$ implies that $Z$ cannot have vertices with zero $\chi$-value, i.e. $Z\subseteq Y_1\cup T_2$ and also in $X$ we can not have any new $A_{\mathbb{F}_4}$ not already in $\Delta$ which implies that for each $1\leq i\leq r$, $Z\cap L_i$ can have at most one vertex. So
$$\iota^{-1}_\leq(Y)=\{X\in\mathcal{T}_\Delta^\chi\mid X\subseteq Y\}=\{\Delta\cup Z\mid \emptyset\neq Z\subseteq Y_1\cup T_2\text{ and }|Z\cap L_i|\leq 1,1\leq i\leq r\}.$$
This poset is isomorphic to 
$$\mathcal{Z}=\{Z\mid  \emptyset\neq Z\subseteq Y_1\cup T_2\text{ and }|Z\cap L_i|\leq 1,1\leq i\leq r\}$$
and has the property that for any $Z\in\mathcal{Z}$, $Z\subseteq Z\cup Y_1\in\mathcal{Z}$. By considering the poset map $\mathcal{Z}\to\mathcal{Z}$ given by $Z\mapsto Z\cup Y_1$ we deduce that $\mathcal{Z}$ is conically contractible (cf. \cite{Benson} Definition 6.4.6).  So using Quillen's Lemma (cf. \cite{Benson} Lemma 6.5.2) we see that $\iota$ induces a homotopy equivalence.
\end{proof}

Finally, we are ready for the last main result that will be needed to prove Theorem \ref{Theorem 1.2}. 
The proof is similar to the proof of  Proposition 4.3 in \cite{Blasco-Cogolludo-Martinez} but we will include it here for the convenience of the reader.
\begin{proposition}\label{Proposition 6.2}
	Let $A_\Gamma$ be an Artin group that satisfies the $K(\pi,1)$-conjecture and $\chi:A_\Gamma\to\mathbb{R}$ be a character. Let $\mathcal{P}$ be the spherical poset and $\mathcal{H}^\chi=\mathcal{P}\setminus\{A_\Delta\mid\Delta\in\mathcal{B}^\chi\}$. Then:
	\begin{enumerate}
		\item If the strong $n$-link condition holds,  the geometric realization of the poset $X=\abs{C(\mathcal{H}^\chi)}$ is $(n-1)$-acyclic,
		\item If the strong homotopical $n$-link condition holds,  the geometric realization of the poset $X=\abs{C(\mathcal{H}^\chi)}$ is $(n-1)$-connected,
	\end{enumerate}
	\end{proposition}	
\begin{proof} We want to use Lemma \ref{Lemma 2.9} where  $h(A_\Delta)=\abs{\Delta}$. For this, let us fix $\Delta\in\mathcal{B}^\chi$ with $h(S)=s$ for $S=A_\Delta$, we need to prove that $\mathrm{Del}^S$ is $(s-2)$-acyclic, that $\abs{\mathcal{J}_S}$ is $(n-1-s)$-acyclic and that $C(\mathcal{P})$ is $(n-1)$-acyclic. The fact that $C(\mathcal{P})$ is $(n-1)$-acyclic holds due to the contractibility of the Deligne complex and the fact that $\mathrm{Del}^S$ is $(s-2)$-acyclic holds due to Lemma \ref{Lemma 2.9}. Hence we only need to prove that $\abs{\mathcal{J}_S}=\abs{\mathcal{J}_\Delta}$ is $(n-1-s)$-acyclic.  By hypothesis, the strong $n$-link condition implies that the geometric realization of $\mathcal{L}^S$ is $(n-1-s)$-acyclic so we only have to use  Proposition \ref{Proposition 6.3}. The same proof works for the homotopical version.
	\end{proof}

This proposition gives us a space $X$ which can be used to apply Theorem \ref{Theorem 2.5}. Our first main result follows easily.

\begin{proof}(Theorem \ref{Theorem 1.2})
	The action of $A_\Gamma$ on $X$ is given by:
	$$g\cdot h(A_{\Delta_0}\subset\dots\subset A_{\Delta_k})=gh(A_{\Delta_0}\subset\dots\subset A_{\Delta_k}).$$
The stabilizers of this action are subgroups of the form $gA_{\Delta}g^{-1}$ where $g\in A_\Gamma$ and $\Delta\subset\Gamma$ is spherical with $\chi(Z(A_\Delta))\neq 0$. Therefore  $A_\Delta\nsubseteq\Ker(\chi)$ and due to Lemma \ref{Lemma 2.4},  $\chi\in \Sigma^{\infty}(A_\Delta)$.  Hence, Lemma \ref{Theorem 2.5} can be applied and the theorem follows.
\end{proof}
\section{Homology of Artin Kernels}\label{Section 6}

Let $A_\Gamma$ be an Artin group. Salvetti (c.f. \cite{Salvetti} and \cite{Salvetti 2}) constructed a space, which will be denoted as $\text{Sal}(\Gamma)$, with  the property that the $K(\pi,1)$-conjecture for $A_\Gamma$ is equivalent to $\text{Sal}(\Gamma)$ being an $A_\Gamma$-classifying space (c.f. \cite{Charney-Davis}). More precisely, the fundamental group of $\text{Sal}(\Gamma)$ is $A_\Gamma$ and the universal cover of $\text{Sal}(\Gamma)$ is contractible if and only if the $K(\pi,1)$-conjecture holds for $A_\Gamma$. 
Recall that an \textbf{Artin kernel} $A_\Gamma^\chi$ is the kernel of some non-trivial discrete character $\chi:A_\Gamma\to\mathbb{Z}$. If $A_\Gamma$ satisfies the $K(\pi,1)$-conjecture, given a commutative ring $R$ and  $\chi:A_\Gamma\to\mathbb{Z}$ a surjective discrete character with kernel $A_\Gamma^\chi$, the homology groups $H_\ast(A_\Gamma^\chi,R)$ are the $R$-homology groups of  the $\chi$-cyclic cover of the Salvetti complex of $A_\Gamma$ (the fact that $\chi$ needs to be surjective is not a problem since by Remark 2.3 of \cite{Blasco-Cogolludo-Martinez 2} we may assume that $\chi$ is surjective).  In \cite{Blasco-Cogolludo-Martinez 2} Blasco-Cogolludo-Martinez computed these homology groups with coefficients in a field for some Artin kernels for even Artin groups of FC-type.

 In this section we will also concentrate in the case when $R=\mathbb{F}$ is a field. Our main objective is to  construct a spectral sequence that one can use to compute the infinite $\mathbb{F}$-dimensional bits of the homology groups $H_\ast(A_\Gamma^\chi,\mathbb{F})$. To do that, we will review some known results about the Salvetti complex and its differential. Later we will use these results  to obtain a normalized version of the chain complex of the cyclic cover associated to a discrete character that will be essential to build the spectral sequence.     Note that the chain complex of this cyclic cover is a chain complex of $R[t^\pm]$ modules where $t$ denotes a generator of the infinite cyclic group $A_\Gamma/A_\Gamma^\chi$.

\subsection{The Salvetti complex and the case of a spherical Artin group}\label{Section 6.1}

The Salvetti complex can be constructed as the $2$-presentation complex  associated to the usual presentation of the Artin group after attaching higher dimensional cells for each spherical $X\subset\Gamma$. In particular its $0$-skeleton is given by a unique cell $\sigma_\emptyset$, its $1$-skeleton is given
	by the $1$-cells $\sigma_v~\forall~v\in V(\Gamma)$ and its $2$-skeleton is given by the $2$-cells $\sigma_e~\forall~e\in E(\Gamma)$. For the next result and throughout the paper, we fix an order on the set of vertices of the defining graph $\Gamma$ which induces an orientation on each clique $X\subseteq\Gamma$. For each $v\in X$ we denote $X_v=X\setminus\{v\}$ and write $\langle X_v\mid X\rangle$ for the incidence of $X_v$ in $X$ which is $1$ or $-1$ according to the orientation given in $X$, namely if $X=\lbrace v_0,\dots,v_k\rbrace$ then $\langle X_{v_i}\mid X\rangle=(-1)^i$.

\begin{proposition}(\cite{Bourbaki} Exercise 3, Chapter 4)
	Let $W_\Gamma$ be a Coxeter group and $\Gamma'\subset\Gamma$. Then, for every class in $W_\Gamma/W_{\Gamma'}$ there exists a unique $w\in W_\Gamma$ with shortest possible word length  in its class $wW_{\Gamma'}$.
\end{proposition}
The elements in the conditions of the previous proposition, i.e. the elements of the shortest possible word length in its class  $wW_{\Gamma'}$, will be called \textbf{$\Gamma'$-reduced} and the set of all $\Gamma'$-reduced elements is denoted by $W_\Gamma^{\Gamma'}$. We set: $$T_\Gamma^{\Gamma'}=\sum_{w\in W_{\Gamma}^{\Gamma'}}(-1)^{l(w)}a_w$$
	where $l(w)$ denotes the length of $w$ and $a_w$ is the lift of $w$ in $A_\Gamma$ (see Definition \ref{def:lift}).
	
\begin{theorem}(\cite{Davis-Leary} Section 7)\label{Davis-Leary}
	Let $A_\Gamma$ be an Artin group and $\widetilde{\text{Sal}(\Gamma)}$ the universal cover of the Salvetti complex of $A_\Gamma$. If $\widetilde{\sigma}_{X}$ represents the lift of the cell $\sigma_X$ in the universal cover, where $X$ is a spherical subgraph of $\Gamma$, then the differentials of the chain complex $C_*(\widetilde{\text{Sal}(\Gamma)})$ are given by:
	$$\partial(\widetilde{\sigma}_{X})=\sum_{v\in X}\langle X_v\mid X\rangle T_X^{X_v}\widetilde{\sigma}_{X_v}.$$
\end{theorem}

Given a discrete character $\chi:A_\Gamma\to\mathbb{Z}$ and a field $\mathbb{F}$,  the associated
 $\mathbb{F}$-chain complex of the $\chi$-cyclic cover of $\text{Sal}(\Gamma)$ can  be computed in the same way as in \cite{Blasco-Cogolludo-Martinez 2}, i.e. as: 
$$\mathbb{F}[t^{\pm1}]\otimes_{\mathbb{F}[A_\Gamma]}C_*(\widetilde{\text{Sal}(\Gamma)})$$

where $\mathbb{F}[t^{\pm1}]$ is seen as an $\mathbb{F}[A_\Gamma]$-module via $1\ast v=t^{\chi(v)}$ and $C_*(\widetilde{\text{Sal}(\Gamma)})$ is, as before, the chain complex of the universal cover of the Salvetti complex. Moreover, if $X\subset\Gamma$ is spherical we will  denote by $\sigma_X^\chi$  the cell associated to $\sigma_X$ in the cyclic cover.\\

Theorem \ref{Davis-Leary}  yields the following expression of the differentials of the cyclic cover
\begin{equation}\label{differential}\partial(\sigma^\chi_{X})=\sum_{v\in X}\langle X_v\mid X\rangle\chi( T_X^{X_v})\sigma^\chi_{X_v}\end{equation}
where 
$$\chi( T_X^{X_v})=\sum_{w\in W_X^{X_v}}(-1)^{l(w)}\chi(a_w).$$

An easy consequence of this expression is the following result that we will need below.

\begin{lemma}\label{Lemma 7.2 1} Let $A_X$ be a spherical irreducible Artin group with cyclic abelianization, i.e., $X\not\in\lbrace\mathbb{B}_m,\mathbb{F}_4,\mathbb{I}_2(2k)\rbrace$ and $0\neq \chi:A_X\to\mathbb{Z}$ a discrete character.  Then $\chi(T_X^{X_v})\neq 0~\forall~v\in X$.
\end{lemma}
\begin{proof}
		Note that in this case $\chi(v)=\chi(u)$ for any $v,u\in X$. By the expression for the differentials (\ref{differential}) we have 
			$$\chi(T_X^{X_v})=\sum\limits_{w\in W_X^{X_v}}(-1)^{l(w)}\chi(a_w)=\sum\limits_{w\in W_X^{X_v}}(-1)^{l(w)}t^{\chi(v)l(w)}$$
			where $W_X^{X_v}$ is the set of coset representatives of shortest length and $l(w)$ is the length of the word $w$.
			 Since the only element of length $0$ in $W_X^{X_v}$ is $1$ and $\chi(v)\neq 0$ we get that the independent term of $\chi(T_X^{X_v})$ is $1$, so $\chi(T_X^{X_v})\neq 0$.
			 \end{proof}

To understand the differentials in (\ref{differential}) for more general spherical $X\subset\Gamma$, we need to understand the $X_v$-reduced elements for every $v\in X$.
 Stumbo explicitly computed these $X_v$-reduced elements for all spherical irreducible Artin groups in \cite{Stumbo}. Using his formulas we can compute the homology with field coefficients of Artin kernels in $A_{\mathbb{I}_2(2k)}$ and $A_{\mathbb{F}_4}$.

\begin{example}\label{Example 4.4} Consider the Artin group $A_{\mathbb{I}_2(2k)}$ with $\mathbb{I}_2(2k)=\lbrace v,w\rbrace$  and an arbitrary character $\chi:A_{\mathbb{I}_2(2k)}\to\mathbb{Z}$. Then,  the boundary maps are:
	$$\partial(\sigma_{v}^\chi)=(1-t^{\chi(v)})\sigma_\emptyset^\chi,~~\partial(\sigma_{w}^\chi)=(1-t^{\chi(w)})\sigma_\emptyset^\chi,$$
	$$\partial(\sigma_{\mathbb{I}_2(2k)}^\chi)=((1-t^{\chi(w)})\sigma_{w}^\chi-(1-t^{\chi(v)})\sigma_{v}^\chi)\frac{t^{k(\chi(v)+\chi(w))}-1}{t^{\chi(v)+\chi(w)}-1}.$$
In particular, if $\chi(v)=1$ and $\chi(w)=-1$ we have:
	$$\partial(\sigma_{v}^\chi)=(1-t)\sigma_\emptyset^\chi,~~\partial(\sigma_{w}^\chi)=(1-t^{-1})\sigma_\emptyset^\chi,$$
	$$\partial(\sigma_{\mathbb{I}_2(2k)}^\chi)=k((1-t^{-1})\sigma_{w}^\chi-(1-t)\sigma_{v}^\chi).$$
So, over a field $\mathbb{F}$ such that $\text{char }\mathbb{F}\mid k$:	
$$H_1(A_{\mathbb{I}_2(2k)}^\chi,\mathbb{F})=\mathbb{F}[t^{\pm 1}](t\sigma_v^\chi+\sigma_w^\chi),$$
which implies that $\dim_{\mathbb{F}}(H_1(A_{\mathbb{I}_2(2k)}^\chi,\mathbb{F}))=\infty$. Using Proposition \ref{Proposition 2.4}, we have given another proof for Lemma \ref{Lemma 3.6}.
\end{example}
\begin{example}\label{Example 4.5}
	Consider the Artin group $A_{\mathbb{F}_4}$ associated to the graph in Example \ref{Examples 3.5} and let $\chi:A_{\mathbb{F}_4}\to\mathbb{Z}$ be an arbitrary character, which must satisfy $\chi(a)=\chi(b)$ and $\chi(c)=\chi(d)$ (because $\overline{a}=\overline{b}$ and $\overline{c}=\overline{d}$ in the abelianization). One can compute the boundary maps and check that:
	$$\partial(\sigma_{abc}^\chi)=-q_c^1q_{a,c}^{1,1}q_{a,c}^{2,1}\sigma_{ab}^\chi+q_c^2q_{a,c}^{1,1}q_{a,c}^{2,1}\sigma_{ac}^\chi-q_a^2q_{a,c}^{2,1}\sigma_{bc}^\chi,$$	
	$$\partial(\sigma_{bcd}^\chi)=-q_a^1q_{a,c}^{1,1}q_{a,c}^{1,2}\sigma_{cd}^\chi+q_a^2q_{a,c}^{1,1}q_{a,c}^{1,2}\sigma_{bd}^\chi-q_c^2q_{a,c}^{1,2}\sigma_{bc}^\chi,$$
	and that $q_{a+c}^2q_{a,c}^{1,1}$ divides $\partial(\sigma_{abcd}^\chi)$,
	where $q_x^1=1-t^{\chi(x)},q_x^2=t^{2\chi(x)}-t^{\chi(x)}+1$ and $q_{x,y}^{r,s}=t^{r\chi(x)+s\chi(y)}+(-1)^{r+s}$. Choosing $\chi(a)=\chi(b)=1$, $\chi(c)=\chi(d)=-1$ and $\mathbb{F}$ a field of characteristic $2$ then the differentials become:
	$$\partial(\sigma_{abc}^\chi)=(t^2+t+1)(t+1)\sigma_{bc}^\chi,$$	
	$$\partial(\sigma_{bcd}^\chi)=(t^{-2}+t^{-1}+1)(t^{-1}+1)\sigma_{bc}^\chi,$$
	$$\partial(\sigma_{abcd}^\chi)=0.$$
	Hence, the third homology group is: 
	$$H_3(A_{\mathbb{F}_4}^\chi,\mathbb{F})=\mathbb{F}[t^{\pm 1}](\sigma_{abc}^\chi+t^3\sigma_{bcd}^\chi).$$
	Thus $\dim_{\mathbb{F}}(H_3(A_{\mathbb{F}_4}^\chi,\mathbb{F}))=\infty$.
\end{example}

In particular using Proposition \ref{Proposition 2.4} we get:
\begin{proposition}\label{negativeF4} If $G=A_{\mathbb{F}_4}$, $\pm[(1,-1)]\not\in\Sigma^j(G,\mathbb{Z})$ for $j\geq 3$.	
\end{proposition}

For the $\mathbb{B}_m$ case we will use the results by Callegaro-Moroni-Salvetti in \cite{Callegaro-Moroni-Salvetti} to give explicit formulas for the minimal length coset representatives of $\mathbb{B}_n$. First, let us fix some notation. Consider the following polynomials on variables $q$, $t$:
$$[m]_q=1+q+\dots+q^{m-1}=\frac{q^m-1}{q-1},$$
$$[m]_q!=\prod_{i=1}^m[i]_q,$$
$$[2m]_{q,t}=[m]_q(1+tq^{m-1}),$$
$$[2m]_{q,t}!!=\prod_{i=1}^m[2i]_{q,t}=[m]_q!\prod_{i=0}^{m-1}(1+tq^{i}).$$
There are two graphs representing the group  $A_{\mathbb{B}_n}$, one is the Coxeter graph $\Gamma$, the other is the Dynkin diagram that we use to label the standard generators as follows
$$
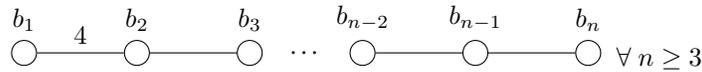
\begin{tikzpicture}[main/.style = {draw, circle},node distance={15mm}] 
	\node[label={$b_1$}][main] (1) {};
	\node[label={$b_2$}][main] (2) [right of=1] {}; 
	\node[label={$b_3$}][main] (3) [right of=2] {}; 
	\draw[-] (1) -- node[above] {$4$}   (2);
	\draw[-] (2) -- (3);
	\node[label={$b_{n-2}$}][main] (4) [right of=3] {};
	\node[label={$b_{n-1}$}][main] (5) [right of=4] {}; 
	\node[label={$b_n$}][main] (6) [right of=5] {}; 
	\draw[-] (4) --  (5);
	\draw[-] (5) -- (6);
	\node at ($(3)!.5!(4)$) {\ldots};
\end{tikzpicture}~\forall~n\geq3$$
 where $q$ represents the generator of $A_{\mathbb{B}_n}/A_{\mathbb{B}_n}'=\mathbb{Z}^2$ associated with the nodes $b_2,\dots,b_n$ and $t$ represents the generator associated with the node $b_1$. In this situation, the coefficients $\chi(T_\Gamma^{\Gamma'})$  in the chain complex of the $\chi$-cyclic are computed as follows:
\begin{itemize}
	\item If $\Delta\subseteq\Gamma$ is a subgraph with connected Dynkin diagram define:
	$$p_{\Delta}(q,t)=\begin{cases} 
		[2m]_{q,t}!!\text{ if }b_1\in V(\Delta)\\
		[m+1]_q!\text{ if }b_1\notin V(\Delta)
	\end{cases}$$
	where $m=\abs{\Delta}$.	
	\item  Let $\Gamma=\Gamma_1\cup\dots\cup\Gamma_n$ and $\Gamma'=\Gamma'_1\cup\dots\cup\Gamma'_m$  be the decompositions of $\Gamma$ and $\Gamma'$ into the connected components of the Dynkin diagram.	
	\item The coefficients $\chi(T_\Gamma^{\Gamma'})$  in the chain complex of the $\chi$-cyclic cover are:
	$$\chi(T_\Gamma^{\Gamma'})=\frac{\prod\limits_{i=1}^np_{\Gamma_i}(-t^{\chi(b_2)},-t^{\chi(b_1)})}{\prod\limits_{i=1}^mp_{\Gamma_{i}'}(-t^{\chi(b_2)},-t^{\chi(b_1)})}.$$
\end{itemize}

With this procedure we are able to prove the following technical result

\begin{lemma}\label{difBm}  Let $X=\mathbb{B}_m,v\in X$ and assume $\chi\neq\pm(1,0)$. Then $\chi(T^{X_v}_X)=0$ if and only if  $\chi=\pm(j-1,-1)$ for $1\leq j\leq m$, $v\in\{b_1,\ldots, b_j\}=\mathcal{B}^\chi(X)$ and either $j$ is odd or $j$ is even and the field $\mathbb{F}$ has characteristic 2. Therefore, if  $\mathbb{F}$ has characteristic 2, $\chi(T^{X_v}_X)=0$ if and only if $v\in\mathcal{B}^\chi(X)$.
\end{lemma}

\begin{proof}Let $\chi$ be an arbitrary character $\neq\pm(1,0)$.  Let us compute $\chi(T^{X_v}_{X})$ using the above method distinguishing two cases. If $v\neq b_1$ we have: 
$$\chi(T^{X_v}_{X})=\frac{p_{\mathbb{B}_m}(-t^{\chi(b_2)},-t^{\chi(b_1)})}{p_{\mathbb{B}_{i-1}}(-t^{\chi(b_2)},-t^{\chi(b_1)})p_{\mathbb{A}_{m-i}}(-t^{\chi(b_2)},-t^{\chi(b_1)})}=\frac{[2m]_{(-t^{\chi(b_2)},-t^{\chi(b_1)})}!!}{[2(i-1)]_{(-t^{\chi(b_2)},-t^{\chi(b_1)})}!![m-i+1]_{(-t^{\chi(b_1)})}!}=$$
$$\frac{[m]_{(-t^{\chi(b_1)})}!}{[i-1]_{(-t^{\chi(b_1)})}![m-i+1]_{(-t^{\chi(b_1)})}!}\frac{\prod\limits_{k=0}^{m-1}(1+(-1)^{k+1}t^{\chi(b_1)+k\chi(b_2)})}{\prod\limits_{k=0}^{i-2}(1+(-1)^{k+1}t^{\chi(b_1)+k\chi(b_2)})}=\chi(T^{X_{b_1,v}}_{X_{b_1}})\prod\limits_{k=i-1}^{m-1}(1+(-1)^{k+1}t^{\chi(b_1)+k\chi(b_2)}).$$
And if $v=b_1$ we have: 
	$$\chi(T^{X_v}_{X})=\chi(T^{X_{b_1}}_{X})=\prod\limits_{k=0}^{m-1}(1+(-1)^{k+1}t^{\chi(b_1)+k\chi(b_2)})$$

Now, as  $X_{b_1}$ is the Coxeter diagram of $\mathbb{A}_{m-1}$ and  $\chi(b_2)\neq 0$, Lemma \ref{Lemma 7.2 1} implies that $\chi(T^{X_{b_1,v}}_{X_{b_1}})\neq0$ for any $v\neq b_1$. Then the only possibly zero factors of $\chi(T^{X_{b_1,v}}_{X_{b_1}})$ correspond to factors of the form $(1+(-1)^{k+1}t^{\chi(b_1)+k\chi(b_2)})$ with  $k=j-1$ and $\chi=\pm(j-1,-1)$ for $1\leq j\leq m$. Moreover $(1+(-1)^jt^{\chi(b_1)+(j-1)\chi(b_2)})$ is a factor of $ \chi(T^{X_v}_{X})$ precisely when $v=b_i$ for $i=1,\dots,j$. This factor is $0$ if  $j$ is odd and $2$ if $j$ is even. Hence, the claim follows.\end{proof}

\subsection{Normalization of the Artin kernel complex}
In this subsection, we consider the chain complex of the cyclic cover of the Salvetti complex and prove that it can be normalized in a way that simplifies the homology computations and preserves the infinite dimensional bits.

Consider an Artin group $A_\Gamma$ satisfying the $K(\pi,1)$-conjecture, $\mathbb{F}$ an arbitrary field and a non-zero discrete character $\chi:A_\Gamma\to\mathbb{Z}$.

Over $\mathbb{F}$ our differentials look like:
$$\partial(\sigma_{X}^\chi)=\sum_{v\in X} \langle X_v\mid X\rangle \chi(T_X^{X_v})\sigma_{X_v}^\chi$$
$\forall~X\subset\Gamma$, where $\chi(T_X^{X_v})\in\mathbb{F}[t^{\pm1}]$ (see (\ref{differential})). Moreover, if $A_X=A_{X_1}\times\dots\times A_{X_n}$ and $v\in X_i$ then $T_X^{X_v}=T_{X_i}^{X_{i,v}}$.

We want to find coefficients $a_X^\chi\in\mathbb{F}[t^{\pm 1}]$ for every spherical subgraph $X\subset\Gamma$ such that we can normalize each $\sigma_X^\chi$ by $a_X$ and get a simpler expression of the differentials. Indeed, we will define $a_X$ as follows:
\begin{itemize}
	\item $a_X:=1$ if $X=\mathcal{B}^\chi(X)$.
	\item $a_X:=\chi(T_X^{X_v})a_{X_v}$ for some $v\in X\setminus\mathcal{B}^\chi(X)$.
\end{itemize}

Let us see some examples and then we will prove that $a_X^\chi$ is well defined.
\begin{examples}\label{Examples 8.1}
	\begin{enumerate}
		\item If $X=\lbrace v\rbrace$ then $\partial(\sigma_X^\chi)=(1-t^{\chi(v)})\sigma_\emptyset^\chi$. Hence: 
		$$a_X^\chi=1-t^{\chi(v)}\text{ if }\chi(v)\neq0\text{ and }a_X^\chi=1\text{ if }\chi(v)=0.$$
		\item If $X=\mathbb{I}_2(2k)$ with $V(X)=\lbrace v,w\rbrace$ then:
		$$\partial(\sigma_X^\chi)=((1-t^{\chi(w)})\sigma_{X_v}^\chi-(1-t^{\chi(v)})\sigma_{X_w}^\chi)\frac{t^{k(\chi(v)+\chi(w))}-1}{t^{\chi(v)+\chi(w)}-1}.$$
		Hence: 
		$$a_X^\chi=1-t^{k\chi(w)}~\text{ if }\chi(v)=0\text{ and }\chi(w)\neq 0,$$
		$$a_X^\chi=1-t^{k\chi(v)}~\text{ if }\chi(w)=0\text{ and }\chi(v)\neq 0,$$
		$$a_X^\chi=1\text{ if }\chi(v)=\chi(w)=0\text{ or }\chi(v)+\chi(w)=0\text{ and }\text{char }\mathbb{F}\mid k,$$
		$$a_X^\chi=\frac{t^{k(\chi(v)+\chi(w))}-1}{t^{\chi(v)+\chi(w)}-1}(1-t^{\chi(v)})(1-t^{\chi(w)})~\text{ otherwise }.$$
				
		\item If $X=\mathbb{I}_2(k)$ with $k$ odd and $V(X)=\lbrace v,w\rbrace$ then:
		$$\partial(\sigma_X^\chi)=\frac{t^{k\chi(v)}+1}{t^{\chi(v)}+1}(\sigma_{X_v}^\chi-\sigma_{X_w}^\chi).$$
		Hence: $$a_X^\chi=\frac{t^{k\chi(v)}+1}{t^{\chi(v)}+1}(1-t^{\chi(v)})\text{ if }\chi(v)\neq0\text{ and }a_X^\chi=1\text{ if }\chi(v)=0.$$
			\end{enumerate}
\end{examples}
%Observe that in example 1  above we could have  $\partial(\sigma_X^\chi)=0$ with $\chi(A_X)\neq 0$ and in example 2 we could have $\chi(A_X)=0$ and $\partial(\sigma_X^\chi)\neq 0$. 

\begin{lemma}\label{Lemma 7.2 2} The element $a_X^\chi$ is well defined. Moreover, if $A_X=A_{X_1}\times\dots\times A_{X_n}$ is the decomposition into irreducible components then $a_X^\chi=a_{X_1}^\chi\dots a_{X_n}^\chi$.
\end{lemma}
\begin{proof}  We prove first that $a_X$ is well defined, i.e., that in the case when  $\mathcal{B}^\chi(X)\subsetneq X$, $a_X^\chi=\chi(T_X^{X_v})a_{X_v}^{\chi}$ does not depend on the $v\in X\setminus\mathcal{B}^\chi(X)$ chosen. Let us argue by induction on $m=\abs{X}$. If $m=1,2$ the result is true as seen in the previous examples. Assume that the result is true for $m-1$. We have to prove that  $\chi(T_X^{X_v})a_{X_v}^\chi=\chi(T_X^{X_w})a_{X_w}^\chi$ for every $v,w\in X\setminus\mathcal{B}^\chi(X)$. By Lemma \ref{Lemma 6.2}.ii) $\mathcal{B}^\chi(X)=\mathcal{B}^\chi(X_v)=\mathcal{B}^\chi(X_w)$ so $v\in X_w\setminus\mathcal{B}^\chi(X_w)$ and $w\in X_v\setminus\mathcal{B}^\chi(X_v)$. Then, by induction hypothesis,  $a_{X_v}=\chi(T_{X_v}^{X_{v,w}})a_{X_{v,w}}^\chi$ and $a_{X_w}=\chi(T_{X_w}^{X_{v,w}})a_{X_{v,w}}^\chi$ are well defined. Moreover, since $\partial^2(\sigma_{X}^\chi)=0$, we deduce that $\chi(T_X^{X_v})\chi(T_{X_v}^{X_{v,w}})=\chi(T_{X}^{X_w})\chi(T_{X_w}^{X_{v,w}})$. This implies that:
			$$a_{X_v}^\chi \chi(T_{X}^{X_v})=a_{X_{v,w}}^\chi \chi(T_{X}^{X_{v}})\chi(T_{X_v}^{X_{v,w}})=a_{X_{v,w}}^\chi \chi(T_{X}^{X_{w}})\chi(T_{X_w}^{X_{v,w}})=a_{X_w}^\chi \chi(T_{X}^{X_w}).$$
 The last claim follows by induction using that  if $v\in X_i$ then $\chi(T_{X}^{X_{v}})=\chi(T_{X_i}^{X_{i,v}})$.
\end{proof}

\begin{lemma}\label{Lemma 7.3} 
	Let $A_\Gamma$ be an Artin group satisfying the $K(\pi,1)$-conjecture, $\chi:A_\Gamma\to\mathbb{Z}$ a non-zero discrete character and $\mathbb{F}$ a field. Let  $X\subseteq\Gamma$ spherical such that $\mathcal{B}^\chi(X)=\Delta\subsetneq X$. Then, for any $v\in X\setminus \Delta$, $\chi(T^{X_v}_X)\neq 0$.
\end{lemma} 
\begin{proof}  If $X=X_1\star\dots\star X_k$ and $v\in X_i$ then $T^{X_v}_X=T^{X_{i,v}}_{X_i}$. Moreover, $\Delta=\mathcal{B}^\chi(X)=\mathcal{B}^\chi(X_1)\star\dots\star\mathcal{B}^\chi(X_k)$ so we may assume that $X$ is irreducible. In the case when 	$X\not\in\lbrace\mathbb{B}_m,\mathbb{F}_4,\mathbb{I}_2(2k)\rbrace$, the statement follows from Lemma \ref{Lemma 7.2 1} and if $X=\mathbb{B}_m$ it follows from Lemma \ref{difBm}. 
	Otherwise:
	\begin{itemize}
\item[i)]	If $X=\mathbb{I}_2(k)=\{v,w\}$\ for $k$ even, we can have $\Delta=\emptyset$ or $\Delta=\{w\}$ with $\chi(w)=0$ and $\chi(v)\neq 0$. 	
\item[ii)]	If $X=\mathbb{F}_4$, using the labelling of Example \ref{Example 4.5}, $\Delta$ could be either empty or one of the edges $\{a,b\}$, $\{d,c\}$ and we would have $\chi(\Delta)=0$ ($\Delta$ can not be $\{b,c\}$ as that would imply that $\{b,c\}$ is 2-dead, thus $\Delta=X$). 
\end{itemize} 
 In either case, the formulas for the differentials (Examples \ref{Example 4.4} and \ref{Example 4.5}) imply $\chi(T^{X_v}_X)\neq 0$ for $v\in X\setminus\Delta$.\end{proof}

Lemma \ref{Lemma 7.3} implies $a_X^\chi\neq 0~\forall~X\subset\Gamma$ spherical. Hence, one can normalize by those coefficients to get an easier expression of the differential maps. For this, we  need to work over a ring where we can divide. To do that we localize  
 the chain complex of the cyclic cover $\mathbb{F}[t^{\pm1}]\otimes_{\mathbb{F}[A_\Gamma]}C_*(\widetilde{\text{Sal}(\Gamma)})$ with respect to $\mathbb{F}[t^{\pm1}]\setminus\lbrace 0\rbrace$ to get a new complex $\mathbb{F}(t^{\pm1})\otimes_{\mathbb{F}[A_\Gamma]}C_*(\widetilde{\text{Sal}(\Gamma)})$. Set:
$$\widetilde{\sigma}_X^\chi:=\frac{1}{a_X}\sigma_X^\chi.$$  
We denote by $\widetilde{C}_\ast(\text{Sal}^\chi_\Gamma)$ this normalized complex. To avoid complications, we will keep the same notation for the differential map in the normalized complex, as $\partial$. In this situation there exist coefficients $f_v\in\mathbb{F}(t^{\pm1})$ satisfying the following expression:
\begin{equation}\label{eq:normalized}\partial(\widetilde{\sigma}_X^\chi)=\sum\limits_{v\in X\setminus\mathcal{B}^\chi(X)}\langle X_v\mid X\rangle\widetilde{\sigma}_{X_v}^\chi+\sum\limits_{v\in \mathcal{B}^\chi(X)}f_v\widetilde{\sigma}_{X_v}^\chi\end{equation}

Since localizing is flat, the rank of the $\mathbb{F}[t^{\pm 1}]$-free part of the homology of the cyclic cover is preserved so we have:\\
\begin{lemma}\label{normalized} 
	Let $A_\Gamma$ be an Artin group satisfying the $K(\pi,1)$-conjecture, $\chi:A_\Gamma\to\mathbb{Z}$ a non-zero discrete character and $\mathbb{F}$ a field. Then the $n$-th homology group $H_n(A_\Gamma^\chi,\mathbb{F})$ of the Artin kernel $\mathrm{Ker}\chi=A_\Gamma^\chi$ is a finitely generated $\mathbb{F}[t^{\pm 1}]$-module whose free part has the same rank as the $\mathbb{F}(t^{\pm 1})$-dimension of the $n$-th homology group of the normalized complex $\widetilde{C}_\ast(\text{Sal}^\chi_\Gamma)$. Therefore the $n$-th homology group $H_n(A_\Gamma^\chi,\mathbb{F})$ has finite $\mathbb{F}$-dimension if and only if the $n$-th homology group of the normalized complex $\widetilde{C}_\ast(\text{Sal}^\chi_\Gamma)$ vanishes.
\end{lemma}
\subsection{A spectral sequence for Artin groups satisfying the $K(\pi,1)$-conjecture}\label{Section 6.3}
In this section we are going to construct a spectral sequence that can be used to determine the homology of the normalized complex $\widetilde{C}_\ast(\text{Sal}^\chi_\Gamma)$ and then, using Lemma \ref{normalized}, we will compute the rank of the $\mathbb{F}[t^{\pm 1}]$-free part of the homology of some Artin kernels. This will be the main tool to prove several partial converses of Theorem \ref{Theorem 1.2}.

Consider an arbitrary field $\mathbb{F}$, an Artin group $A_\Gamma$ satisfying the $K(\pi,1)$-conjecture and a non-zero discrete character $\chi:A_\Gamma\to\mathbb{Z}$. Let us define a filtration of the cyclic cover of the Salvetti complex associated to $\chi$. Recall that $\mathcal{B}^\chi(X)=\max\lbrace Y \subset X\mid\chi(Z(A_Y))=0\rbrace$ for every $X\subset\Gamma$ spherical
and let: $$F_k(\text{Sal}^\chi)=\oplus\lbrace \mathbb{F}(t^{\pm 1})\widetilde{\sigma}_{X}^\chi\mid X\subset\Gamma\text{ is spherical and }\abs{\mathcal{B}^\chi(X)}\leq k\rbrace~\forall~k\geq 0.$$
This yields a well defined filtration due to Lemma \ref{Lemma 6.2} i): $$\emptyset=F_{-1}(\text{Sal}^\chi)\subset F_0(\text{Sal}^\chi)\subset F_1(\text{Sal}^\chi)\subset\dots\subset F_N(\text{Sal}^\chi)=\text{Sal}^\chi$$ 
(Here, $N$ is some number big enough)

Given a finite filtration of a chain complex there is a canonical spectral sequence converging to the homology of the complex (\cite{Weibel} Section 5.4 and \cite{Hatcher}).

\begin{proposition}\label{Proposition 8.3} There is a spectral sequence converging to the homology of the normalized complex $\widetilde{C}_\ast(\text{Sal}^\chi_\Gamma)$ that can be used to compute the infinite dimensional components of $H_n(A_\Gamma^\chi,\mathbb{F})$. The
 only possibly non vanishing terms in the 1-page of this spectral sequence are
	$$E^1_{p,q}=\mathbb{F}(t^{\pm1})\otimes\left(\bigoplus\limits_{\substack{\Delta\in\mathcal{B}^\chi\\ \abs{\Delta}=p}}\overline{H}_{q-1}(L^\chi_\Delta)\right)~\forall~p=0,\dots,N\text{ and }q\geq 0 $$ 
	where $\overline{H}$ denotes reduced homology and $L_\Delta^\chi$ is the simplicial complex from in Definition \ref{Definition 6.4}. Moreover
	\begin{itemize}
	
\item[i)] If $\chi(v)\neq 0~\forall~v\in\Gamma$, $\text{char }\mathbb{F}=p$ and $p\mid\frac{l(e)}{2}$ for every edge $e\in\Gamma$ with even label $l(e)\geq 4$, then the spectral sequence collapses at page $1$.
\item[ii)] If $A_\Gamma$ is even, $\text{char }\mathbb{F}=p$ and $p\mid\frac{l(e)}{2}$ for every edge $e\in\Gamma$ with label $l(e)\geq 4$ then the spectral sequence collapses at page 1.
\end{itemize}
\end{proposition}
\begin{proof}  Consider the spectral sequence of the filtration above, i.e., the 0-page is:
$$E^0_{p,\ast}=F_p(\text{Sal}^\chi)/F_{p-1}(\text{Sal}^\chi)=\oplus\lbrace \mathbb{F}(t^{\pm 1})\overline{\sigma}_{X}^\chi\mid X\subset\Gamma\text{ is spherical and }\abs{\mathcal{B}^\chi(X)}=p\rbrace~\forall~p=0,\dots,N$$
where $\overline{\sigma}_{X}^\chi$ represents the equivalence class of $\widetilde{\sigma}_{X}^\chi$ in the quotient.  We describe now the 1-page.
For any $X\subseteq\Gamma$ spherical such that $\abs{\mathcal{B}^\chi(X)}=p$, if $v\in\mathcal{B}^\chi(X)$, then $\abs{\mathcal{B}^\chi(X_v)}<p$ so using (\ref{eq:normalized}) we see that in the quotient      
		 $F_p(\text{Sal}^\chi)/F_{p-1}(\text{Sal}^\chi)$ the differentials are given by: 
		$$\partial(\overline{\sigma}_{X}^\chi)=\sum\limits_{v\in X\setminus \Delta}\langle X_v\mid X\rangle\overline{\sigma}_{X_v}^\chi.$$
				Moreover, for $X\subseteq\Gamma$ spherical, $\Delta=\mathcal{B}^\chi(X)$, if and only if $X\in\mathcal{T}^\chi_\Delta$ by definition (see Proposition \ref{Proposition 6.3}). For every$A_\Delta\in\mathcal{B}^\chi$ with $\abs{\Delta}=p$ we define:  $$\left(F_p(\text{Sal}^\chi)/F_{p-1}(\text{Sal}^\chi)\right)_\Delta=\oplus\lbrace \mathbb{F}(t^{\pm 1})\overline{\sigma}_{X}^\chi\mid X\in \mathcal{T}^\chi_\Delta\rbrace~\forall~k=0,\dots,N$$ 
		This implies that:
		$$F_p(\text{Sal}^\chi)/F_{p-1}(\text{Sal}^\chi)=\bigoplus_{\substack{A_\Delta\in\mathcal{B}^\chi\\\abs{\Delta}=p}}\left(F_p(\text{Sal}^\chi)/F_{p-1}(\text{Sal}^\chi)\right)_\Delta$$
		For any $X\in\mathcal{T}^\chi_\Delta$, $X=\Delta\cup Y$ where $Y$ is a $(\abs{X}-p-1)$-cell of  $L_\Delta^\chi$ (see Proposition \ref{Proposition 6.3}), which means that:
		$$C_n((F_p(\text{Sal}^\chi)/F_{p-1}(\text{Sal}^\chi))_\Delta)\cong\overline{C}_{n-\abs{p}-1}\left(L_\Delta^\chi\right).$$
		where $\overline{C}$ is the augmented chain complex. Taking homology groups we deduce that the term $E^1_{p,q}$ is as in the statement.\\

 For i) and ii) we claim that in these conditions  $\chi(T_X^{X_v})=0$ for each $v\in\Delta=\mathcal{B}^\chi(X)$ which means that in the spectral sequence all the induced differentials from page 1 onwards vanish. We may assume that $X$ is irreducible and that $\mathcal{B}^\chi(X)\neq\emptyset$. Then, the hypothesis i) implies that $X\in\lbrace\mathbb{I}(2k),\mathbb{F}_4,\mathbb{B}_m\rbrace$ and the hypothesis ii) implies that $X=\{v\}$ with $\chi(v)=0$ or $X=\mathbb{I}(2k)$. If  $X\in\lbrace\mathbb{F}_4,\mathbb{B}_m\rbrace$, the prime $p$ in the hypothesis must be 2 and then the claim follows by either the formulas in Example \ref{Example 4.5} for the $\mathbb{F}_4$ case or by Lemma \ref{difBm} in the $\mathbb{B}_m$ case. If $X=\{v\}$ with $\chi(v)=0$ or $X=\mathbb{I}(2k)$  we get the same but using the formulas in Example \ref{Example 4.4}, note that if $X=\mathbb{I}(2k)$ the hypothesis implies that $p$ must divide $k$.\end{proof}

In the next section, we will see how to use the spectral sequence together with Theorem \ref{Theorem 1.2} to fully compute the $\Sigma$-invariants in some cases. However this is not always possible, as sometimes the homology computations are not enough to prove that a given character is not in a $\Sigma$-invariant: the problem is that the corresponding $n$-homology group could be finite dimensional (or have finite rank if we work with integer coefficients) but without the kernel being of type $FP_n$. In fact this can happen for very simple graphs, as the following examples show.
 
\begin{examples} \begin{enumerate}
\item Consider $A_\Gamma$ where $\Gamma$ is the following graph:
$$
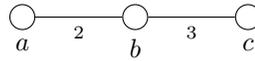
\begin{tikzpicture}[main/.style = {draw, circle},node distance={15mm}] 
	\node[label=below:{$a$}][main] (1)  {}; 
		\node[label=below:{$b$}][main] (2) [right of=1] {}; 
	\node[label={below:$c$}][main] (3) [right of=2] {};
	
		\draw[-] (1) -- node[below] {$\scriptstyle{2}$}   (2);
	\draw[-] (2) -- node[below] {$\scriptstyle{3}$}   (3);
\end{tikzpicture} $$
and $\chi:A_\Gamma\to\mathbb{Z}$ with $\chi(a)=1$, $\chi(b)=\chi(c)=0$. Then $\mathcal{B}^\chi=\{\emptyset,b,c,bc\}$ and $L_{c}^\chi=L_{bc}^\chi=\emptyset$. Then the only non-vanishing terms in the 1-page of the spectral sequence are
$$E^1_{1,0}=\overline{H}_{-1}(L_c^\chi)=\mathbb{F},$$
$$E^1_{2,0}=\overline{H}_{-1}(L_{bc}^\chi)=\mathbb{F}$$
and it is easy to check that $d^1_{2,0}:E^1_{2,0}\to E^1_{1,0}$ is the identity map. Hence, the homology of the normalized complex vanishes at degrees 1, 2, which means that $H_1(\Ker(\chi),\mathbb{F})$ and $H_2(\Ker(\chi),\mathbb{F})$ have finite $\mathbb{F}$-dimension. Indeed, one can check that the abelianization of $\ker\chi$ is $H_1(\Ker(\chi))=\mathbb{Z}$. However, using the $\Sigma^1$-conjecture which is true in this case, we see that $[\chi]\not\in\Sigma^1(A_\Gamma)$ so $\Ker(\chi)$ is not finitely generated.

\item Let $\Gamma$ be the following graph
$$\begin{tikzpicture}[main/.style = {draw, circle},node distance={15mm}] 
	\node[label=left:{$a$}][main] (1)  {}; 
		\node[label=right:{$b$}][main] (2) [right of=1] {}; 
	\node[label=left:{$c$}][main] (3) [above of=1] {};
\node[label=right:{$d$}][main] (4) [above of=2] {};
	
		\draw[-] (1) -- node[below] {$\scriptstyle{4}$}   (2);
	\draw[-] (2) -- node[right] {$\scriptstyle{2}$}   (4);
	\draw[-] (1) -- node[left] {$\scriptstyle{2}$}   (3);
\draw[-] (3) -- node[above] {$\scriptstyle{6}$}   (4);

\end{tikzpicture} $$
and $\chi:A_\Gamma\to\mathbb{Z}$ with $\chi(a)=\chi(c)=1$, $\chi(b)=\chi(d)=-1$. Now, $\mathcal{B}^\chi=\{\emptyset,ab,cd\}$, $L_{ab}^\chi=L_{cd}^\chi=\emptyset$ and $L_\emptyset^\chi$ consists of the two isolated edges $ca$ and $db$. So
$$E^1_{2,0}=\overline{H}_{-1}(L_{ab}^\chi)\oplus\overline{H}_{-1}(L_{cd}^\chi)=\mathbb{F},$$
$$E^1_{0,1}=\overline{H}_{0}(L_{\emptyset}^\chi)=\mathbb{F}.$$
Here, $d^1=0$ but the differential  $d^2:E^2_{2,0}\to E^1_{0,1}$ does not vanish. In fact, working in characteristic 2 or 3 then $d^2$ vanishes in one of the summands but not in the other. Hence $\dim_\mathbb{F}H_1(\ker(\chi),\mathbb{F})<\infty$ but, as before, as the $\Sigma^1$-conjecture holds true in this case, we see that  $[\chi]\not\in\Sigma^1(A_\Gamma)$ so $\ker(\chi)$ is not finitely generated.

\end{enumerate}
\end{examples}

\section{Applications}\label{Section 7}

We begin by proving some partial converses to Theorem \ref{Theorem 1.2}. In particular, the next theorem yields Theorem \ref{BestvinaBradyArtin} in the introduction,  which is the generalization for  Artin groups of the homological part of the well-known Bestvina-Brady Theorem for right angled Artin groups.

\begin{theorem}\label{Theorem 8.4}
	Let $A_\Gamma$ be an Artin group satisfying the $K(\pi,1)$-conjecture and $0\neq\chi:A_\Gamma\to\mathbb{Z}$ be a character such that $\mathcal{B}^\chi=\{\emptyset\}$, for example any character with $\chi(v)>0$ for all $v\in\Gamma$. Then, $[\chi]\in\Sigma^n(A_\Gamma,\mathbb{Z})$ if and only if the complex obtained from $\Gamma$ by adding an $(m-1)$-cell to every spherical subset $X\subseteq\Gamma$ with $\abs{X}=m$ is $(n-1)$-acyclic.
	\end{theorem}
	\begin{proof}
		Since $\mathcal{B}^\chi=\{\emptyset\}$, Proposition \ref{Proposition 8.3} implies that $E^1_{0,q}=\mathbb{F}(t^{\pm1})\otimes_{\mathbb{F}[A_\Gamma]}\overline{H}_{q-1}(L^\chi_\emptyset)$ if $q\geq 0$ and $E_{p,q}^1=0$ if $p\neq 0$. Then, $E^r_{p,q}=E^1_{p,q}~\forall~r\geq1$ which means the following:
		$$H_{n}(A_\Gamma^\chi,\mathbb{F})\cong\bigoplus\limits_{p\in\mathbb{Z}} E^\infty_{p,n-p}=E^{1}_{0,q}=\mathbb{F}(t^{\pm1})\otimes_{\mathbb{F}[A_\Gamma]}\overline{H}_{q-1}(L^\chi_\emptyset)~\forall~n\geq0$$
		Therefore, taking dimensions and using Proposition \ref{Proposition 2.4} and Theorem \ref{Theorem 1.2} we deduce that $[\chi]\in\Sigma^n(A_\Gamma,\mathbb{Z})$ if and only if $\chi$ satisfies the strong $n$-link condition which in this case is equivalent to saying that $L_\emptyset^\chi$ is $(n-1)$-acyclic. But note that the complex $L_\emptyset^\chi$ is just the complex in the statement. 
	\end{proof}
	
\begin{theorem}\label{Theorem 8.5}
	Let $A_\Gamma$ be an Artin group satisfying the $K(\pi,1)$-conjecture such that there exists a prime number $p$ with $p\mid\frac{l(e)}{2}$ for any $e\in \Gamma$ with even $l(e)>2$ and $0\neq\chi: A_\Gamma\mapsto\mathbb{Z}$ be a character such that $\chi(v)\neq 0~\forall~v\in\Gamma$. Then, $[\chi]\in\Sigma^n(A_\Gamma,\mathbb{Z})$ if and only if $\chi$ satisfies the strong $n$-link condition.
\end{theorem}
\begin{proof}
	By Proposition \ref{Proposition 8.3} ii) the spectral sequence collapses at page $1$. Hence:
	$$H_{n}(A_\Gamma^\chi,\mathbb{F})\cong\bigoplus\limits_{p\in\mathbb{Z}} E^\infty_{p,n-p}=\bigoplus\limits_{p=0}^N E^{1}_{p,n-p}=\left(\bigoplus\limits_{p=0}^N\mathbb{F}(t^{\pm1})\otimes_{\mathbb{F}[A_\Gamma]}\bigoplus\limits_{\substack{\Delta\in\mathcal{B}^\chi\\ \abs{\Delta}=p}}\overline{H}_{n-p-1}(L^\chi_\Delta)\right)~\forall~n\geq0$$
	Thus, $\dim_{\mathbb{F}}H_n(A_\Gamma^\chi,\mathbb{F})=\infty$ if and only if $\chi$ does not satisfy the strong $n$-link condition, so Proposition \ref{Proposition 2.4} and Theorem \ref{Theorem 1.2}  prove the result.
\end{proof}

\begin{remark}\label{after8.5} Note that the same argument implies that if there exists a prime number $p$ with $p\mid\frac{l(e)}{2}$  for any $e\in \Gamma$ with even $l(e)>2$  and $0\neq\chi: A_\Gamma\mapsto\mathbb{Z}$ is a character such that for any $v\in\Gamma$ with $\chi(v)=0$, $v$ does not belong to any edge with odd label, then $[\chi]\in\Sigma^n(A_\Gamma,\mathbb{Z})$ if and only if $\chi$ satisfies the strong $n$-link condition.
\end{remark}

Applying Theorem \ref{Theorem 8.5} for $n=1$ we prove Theorem \ref{Theorem 1.4}.
\begin{proof}(Theorem \ref{Theorem 1.4})
	As indicated in section \ref{Strong}, we need to prove that if $\chi: A_\Gamma\to\mathbb{Z}$ is a non-zero discrete character with $\text{Liv}^\chi$ disconnected then $[\chi]\notin\Sigma^1(A_\Gamma)$ because then we get the same for arbitrary characters using the fact that $\Sigma^1$ is open and the discrete characters are dense. Moreover, we do not need to assume the $K(\pi,1)$-conjecture here as the 2-skeleton of the Salvetti complex is the presentation complex. By contradiction, assume that $[\chi]\in\Sigma^1(A_\Gamma)$. Then $[\chi]\in\Sigma^1(A_{\text{Liv}_0^\chi})$ by Proposition \ref{Proposition 2.7} and the living subgraphs of both $\Gamma$ and $\text{Liv}_0^\chi$ coincide. Hence, we may assume that $\Gamma=\text{Liv}_0^\chi$. Applying Theorem \ref{Theorem 8.5}, we see that $\chi$ satisfies the strong $1$-link condition. But the strong $1$-link condition in this case, since there is no $v\in\Gamma$ with $\chi(v)=0$, means that $L^\chi_\emptyset$ is $0$-acyclic, i.e. $L^\chi_\emptyset$ is connected. But $L^\chi_\emptyset$ is connected iff $\text{Liv}^\chi$ is connected, which is a contradiction.
\end{proof}

We can also easily complete the computation of the Sigma invariants for spherical Artin groups
\begin{proof}(Theorem \ref{Theorem 1.1})
We only need to consider the groups $\mathbb{F}_4$ and $\mathbb{B}_m$. Both satisfy the hypothesis of Theorem \ref{Theorem 8.5} for the prime $p=2$ so the result follows from Examples \ref{Examples 3.5},  Examples \ref{strongF4Bm}, Theorem \ref{Theorem 1.2} and   Theorem \ref{Theorem 8.5}.  
\end{proof}

As a consequence of this description of the $\Sigma$-invariants of irreducible spherical Artin groups it is easy to deduce  that for any  spherical  Artin group $G=A_\Delta$
	$$\Sigma^\infty(G)=\Sigma^\infty(G,\mathbb{Z})=\lbrace[\chi]\in S(A_\Delta)\mid\chi(Z(A_\Delta))\neq 0\rbrace.$$

For even Artin groups satisfying the $K(\pi,1)$-conjecture,  the condition on the existence of a suitable prime $p$ alone allows as to determine the homological $\Sigma$-invariants in terms of the homological strong $n$-link condition.

\begin{theorem}
	Let $A_\Gamma$ be an even Artin group satisfying the $K(\pi,1)$-conjecture such that there exists a prime number $p$ with $p\mid\frac{l(e)}{2}$  for any $e\in \Gamma$ with even $l(e)>2$  and $0\neq\chi:A_\Gamma\mapsto\mathbb{R}$ a character. Then, $[\chi]\in\Sigma^n(A_\Gamma,\mathbb{Z})$ if and only if $\chi$ satisfies the strong $n$-link condition.
\end{theorem}
\begin{proof}
	For the discrete case, if $\chi:A_\Gamma\to\mathbb{Z}$ is a non-zero character, then by Proposition \ref{Proposition 8.3} iii) the spectral sequence collapses at page $1$. Hence, arguing in the same way as in Theorem \ref{Theorem 8.5} the result follows. For the arbitrary case, use that the $\Sigma$-invariants are open and the discrete characters are dense.
	\end{proof}

\subsection{Affine Artin groups}
Recall that irreducible affine Coxeter groups are classified (cf. \cite{Humphreys} Section 4.7). There are $4$ families of groups and $6$ sporadic groups that are represented by their Dynkin diagram as follows:
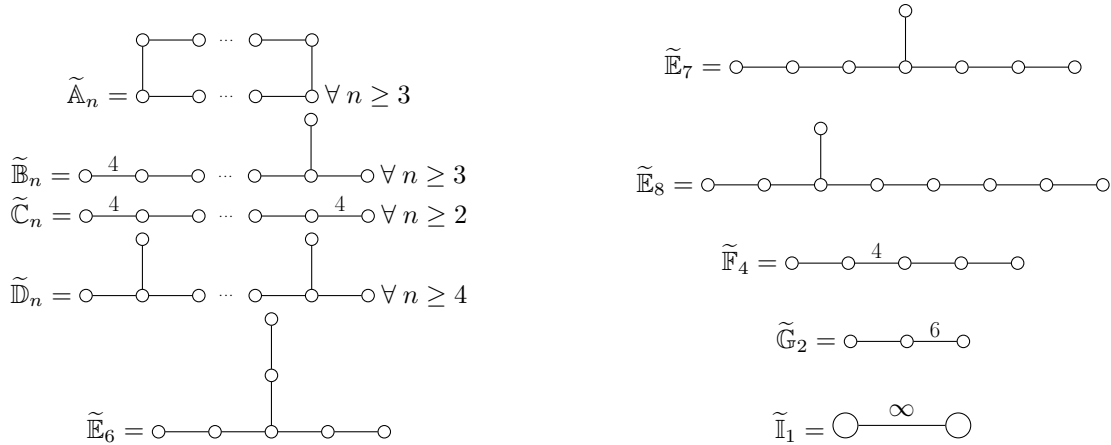
\begin{multicols}{2}
	$$\widetilde{\mathbb{A}}_n=
\begin{tikzpicture}[main/.style = {draw, circle},node distance={15mm},scale=0.6, every node/.style={scale=0.5}] 
	\node[main] (1) {};
	\node[main] (2) [right of=1] {}; 
	\node[main] (3) [above of=1] {}; 
	\node[main] (4) [right of=3] {}; 
	\draw[-] (1) --  (2);
	\draw[-] (1) -- (3);
	\draw[-] (3) -- (4);
	\node[main] (5) [right of=2] {}; 
	\node[main] (6) [right of=4] {}; 
	\node[main] (7) [right of=5] {};
	\node[main] (8) [right of=6] {};
	\node at ($(2)!.5!(5)$) {\ldots};
	\node at ($(4)!.5!(6)$) {\ldots};
	\draw[-] (5) -- (7);
	\draw[-] (6) -- (8);
	\draw[-] (7) -- (8);
\end{tikzpicture}~\forall~n\geq3$$
$$\widetilde{\mathbb{B}}_n=\begin{tikzpicture}[main/.style = {draw, circle},node distance={15mm},scale=0.6, every node/.style={scale=0.5}]
	\node[main] (1) {};
	\node[main] (2) [right of=1] {}; 
	\node[main] (3) [right of=2] {}; 
	\draw[-] (1) -- node[above] {$\mathlarger{\mathlarger{\mathlarger{\mathlarger{4}}}}$}   (2);
	\draw[-] (2) -- (3);
	\node[main] (4) [right of=3] {};
	\node[main] (5) [right of=4] {}; 
	\node[main] (6) [right of=5] {}; 
	\node[main] (7) [above of=5] {};
	\draw[-] (4) --  (5);
	\draw[-] (5) -- (6);
	\draw[-] (5) -- (7);
	\node at ($(3)!.5!(4)$) {\ldots};
\end{tikzpicture}~\forall~n\geq3$$
$$\widetilde{\mathbb{C}}_n=\begin{tikzpicture}[main/.style = {draw, circle},node distance={15mm},scale=0.6, every node/.style={scale=0.5}] 
	\node[main] (1) {};
	\node[main] (2) [right of=1] {}; 
	\node[main] (3) [right of=2] {}; 
	\draw[-] (1) -- node[above] {$\mathlarger{\mathlarger{\mathlarger{\mathlarger{4}}}}$}   (2);
	\draw[-] (2) -- (3);
	\node[main] (4) [right of=3] {};
	\node[main] (5) [right of=4] {}; 
	\node[main] (6) [right of=5] {}; 
	\draw[-] (4) --  (5);
	\draw[-] (5) -- node[above] {$\mathlarger{\mathlarger{\mathlarger{\mathlarger{4}}}}$}   (6);
	\node at ($(3)!.5!(4)$) {\ldots};
\end{tikzpicture}~\forall~n\geq2$$
$$\widetilde{\mathbb{D}}_n=
\begin{tikzpicture}[main/.style = {draw, circle},node distance={15mm},scale=0.6, every node/.style={scale=0.5}] 
	\node[main] (1) {};
	\node[main] (2) [right of=1] {}; 
	\node[main] (3) [right of=2] {}; 
	\node[main] (4) [above of=2] {}; 
	\draw[-] (1) -- (2);
	\draw[-] (2) -- (3);
	\draw[-] (2) -- (4);
	\node[main] (5) [right of=3] {};
	\node[main] (6) [right of=5] {}; 
	\node[main] (7) [right of=6] {}; 
	\node[main] (8) [above of=6] {};
	\draw[-] (5) --  (6);
	\draw[-] (6) -- (7);
	\draw[-] (6) -- (8);
	\node at ($(3)!.5!(5)$) {\ldots};
\end{tikzpicture}~\forall~n\geq4$$
$$\widetilde{\mathbb{E}}_6=
\begin{tikzpicture}[main/.style = {draw, circle},node distance={15mm},scale=0.6, every node/.style={scale=0.5}] 
	\node[main] (1) {};
	\node[main] (2) [right of=1] {}; 
	\node[main] (3) [right of=2] {}; 
	\node[main] (4) [above of=3] {}; 
	\node[main] (5) [right of=3] {}; 
	\node[main] (6) [right of=5] {}; 
	\node[main] (7) [above of=4] {}; 
	\draw[-] (1) -- (2);
	\draw[-] (2) -- (3);
	\draw[-] (3) -- (4);
	\draw[-] (3) -- (5);
	\draw[-] (5) -- (6);
	\draw[-] (4) -- (7);
\end{tikzpicture}$$
$$\widetilde{\mathbb{E}}_7=
\begin{tikzpicture}[main/.style = {draw, circle},node distance={15mm},scale=0.6, every node/.style={scale=0.5}] 
	\node[main] (1) {};
	\node[main] (2) [right of=1] {}; 
	\node[main] (3) [right of=2] {}; 
	\node[main] (4) [above of=5] {}; 
	\node[main] (5) [right of=3] {}; 
	\node[main] (6) [right of=5] {}; 
	\node[main] (7) [right of=6] {}; 
	\node[main] (8) [right of=7] {}; 
	\draw[-] (1) -- (2);
	\draw[-] (2) -- (3);
	\draw[-] (5) -- (4);
	\draw[-] (3) -- (5);
	\draw[-] (5) -- (6);
	\draw[-] (6) -- (7);
	\draw[-] (7) -- (8);
\end{tikzpicture}$$
$$\widetilde{\mathbb{E}}_8=
\begin{tikzpicture}[main/.style = {draw, circle},node distance={15mm},scale=0.6, every node/.style={scale=0.5}] 
	\node[main] (1) {};
	\node[main] (2) [right of=1] {}; 
	\node[main] (3) [right of=2] {}; 
	\node[main] (4) [above of=3] {}; 
	\node[main] (5) [right of=3] {}; 
	\node[main] (6) [right of=5] {}; 
	\node[main] (7) [right of=6] {}; 
	\node[main] (8) [right of=7] {}; 
	\node[main] (9) [right of=8] {}; 
	\draw[-] (1) -- (2);
	\draw[-] (2) -- (3);
	\draw[-] (3) -- (4);
	\draw[-] (3) -- (5);
	\draw[-] (5) -- (6);
	\draw[-] (6) -- (7);
	\draw[-] (7) -- (8);
	\draw[-] (8) -- (9);
\end{tikzpicture}$$
$$\widetilde{\mathbb{F}}_4=
\begin{tikzpicture}[main/.style = {draw, circle},node distance={15mm},scale=0.6, every node/.style={scale=0.5}] 
	\node[main] (1) {};
	\node[main] (2) [right of=1] {}; 
	\node[main] (3) [right of=2] {}; 
	\node[main] (4) [right of=3] {}; 
	\node[main] (5) [right of=4] {}; 
	\draw[-] (1) -- (2);
	\draw[-] (2) -- node[above] {$\mathlarger{\mathlarger{\mathlarger{\mathlarger{4}}}}$} (3);
	\draw[-] (3) -- (4);
	\draw[-] (4) -- (5);
\end{tikzpicture}$$
$$\widetilde{\mathbb{G}}_2=
\begin{tikzpicture}[main/.style = {draw, circle},node distance={15mm},scale=0.6, every node/.style={scale=0.5}] 
	\node[main] (1) {};
	\node[main] (2) [right of=1] {}; 
	\node[main] (3) [right of=2] {}; 
	\draw[-] (1) -- (2);
	\draw[-] (2) -- node[above] {$\mathlarger{\mathlarger{\mathlarger{\mathlarger{6}}}}$} (3);
\end{tikzpicture}$$
$$\widetilde{\mathbb{I}}_1=
\begin{tikzpicture}[main/.style = {draw, circle},node distance={15mm}],scale=0.6, every node/.style={scale=0.5} 
	\node[main] (1) {};
	\node[main] (2) [right of=1] {}; 
	\draw[-] (1) -- node[above] {$\infty$} (2);
\end{tikzpicture}$$ 
\end{multicols}
In the figure above we are using the Dynkin diagram notation, so the $\infty$ label means that there is no relation between those two generators. As for the spherical case, an affine Artin group is just an Artin group whose associated Coxeter group is an affine Coxeter group.

To begin with, it is obvious to compute the character sphere of affine Artin groups.
\begin{lemma}
	\begin{itemize}
		\item[i)] 	If $G\in\lbrace A_{\widetilde{\mathbb{A}}_n},A_{\widetilde{\mathbb{D}}_n},A_{\widetilde{\mathbb{E}}_6},A_{\widetilde{\mathbb{E}}_7},A_{\widetilde{\mathbb{E}}_8}\rbrace$ then $S(G)=\mathbb{S}^0$.
		\item[ii)] If $G\in\lbrace A_{\widetilde{\mathbb{B}}_n},A_{\widetilde{\mathbb{F}}_4},A_{\widetilde{\mathbb{G}}_2},A_{\widetilde{\mathbb{I}}_1}\rbrace$ then $S(G)=\mathbb{S}^1$.
		\item[iii)] If $G=A_{\widetilde{\mathbb{C}}_n}$ then $S(G)=\mathbb{S}^2$.
	\end{itemize}
\end{lemma}

To study the $\Sigma$-invariants we need to understand when the strong $n$-link condition holds. In the case of affine Artin groups, the relevant simplicial complexes can be described in terms of the geometric realization of subposets of $\mathrm{Bol}_n$, the finite Boolean lattice of subsets of $\lbrace 1,\dots,n\rbrace$ ordered by containment. Some relevant facts about the geometric realization of $\mathrm{Bol}_n$ can be found at \cite{Wachs}, but we will include the basic facts that we will be using. Denote by  $\hat{0}$ and $\hat{1}$ the minimal and maximal elements respectively. Recall that the geometric realization of a poset with a maximal or a minimal element is always contractible and that the geometric realization of $\mathrm{Bol}_n\setminus\lbrace \hat{0},\hat{1}\rbrace$ is homotopic to $\mathbb{S}^{n-2}$. The direct product of two posets $P$ and $Q$ is defined as the poset $P\times Q$ whose underlying set is the cartesian product $\lbrace(p,q)\mid p\in P,q\in Q\rbrace$ and whose order relation is given by $(p_1,q_1)\leq_{P\times Q}(p_2,q_2)\Leftrightarrow p_1\leq_{P}p_2$ and $q_1\leq_{Q}q_2$. The geometric realization of $P\times Q$ satisfies $\abs{P\times Q}\cong\abs{P}\times\abs{Q}$ and $\abs{(P\times Q)\setminus\lbrace\hat{0}\rbrace}\cong\abs{P\setminus\lbrace\hat{0}\rbrace}\star\abs{Q\setminus\lbrace\hat{0}\rbrace}$ (cf. \cite{Walker} Theorem 5.1). With this notation we have:
\begin{lemma}\label{Lemma 9.4}
	Let $n$ be a positive integer and $T\in \mathrm{Bol}_n$ with $1\leq\abs{T}<n$. Then the geometric realization of $A=\lbrace\emptyset\neq U\in \mathrm{Bol}_n\mid T\nsubseteq U\rbrace$ is contractible.
\end{lemma}
\begin{proof} Let $\abs{T}=k$ and $T_1=\{1,\ldots,n\}\setminus T$. For any $U\in A$, $U=(U\cap T)\cup(U\cap T_1)$ so we can identify 
$$A=((\mathrm{Bol}_k\setminus\lbrace\hat{1}\rbrace)\times \mathrm{Bol}_{n-k})\setminus\lbrace\hat{0}\rbrace$$ and $\abs{A}=\abs{\mathrm{Bol}_k\setminus\lbrace\hat{0},\hat{1}\rbrace}\star \abs{\mathrm{Bol}_{n-k}\setminus\lbrace\hat{0}\rbrace}.$
 The poset $\mathrm{Bol}_{n-k}\setminus\lbrace\hat{0}\rbrace$ has a minimum so its geometric realization is contractible thus $|A|$ also is.
\end{proof}

\begin{proposition}\label{Proposition 9.5}
	Let $n$ be a positive integer $\geq 3$ and $2\leq j,k\leq n-1$. Define: $$A_{j,k}=\lbrace\emptyset\neq U\in \mathrm{Bol}_n\mid (1,\dots,j)\nsubseteq U,(n-k,\dots,n)\nsubseteq U\rbrace$$
	If $j\leq n-k-2$ then the geometric realization of $A_{j,k}$ is contractible, while if $j\geq n-k-1$ then the geometric realization of $A_{j,k}$ is homotopic to $\mathbb{S}^{n-3}$.
\end{proposition}
\begin{proof} If $j\leq n-k-1$ then: 
		$$|A_{j,k}|\simeq|((\mathrm{Bol}_{j}\setminus\lbrace\hat{1}\rbrace)\times \mathrm{Bol}_{n-j-k-1}\times (\mathrm{Bol}_{k+1}\setminus\lbrace\hat{1}\rbrace))\setminus\lbrace\hat{0}\rbrace|=|\mathrm{Bol}_{j}\setminus\lbrace\hat{0},\hat{1}\rbrace|\star| \mathrm{Bol}_{n-j-k-1}\setminus\lbrace\hat{0}\rbrace|\star |\mathrm{Bol}_{k+1}\setminus\lbrace\hat{0},\hat{1}\rbrace|\simeq$$$$\simeq\mathbb{S}^{j-2}\star | \mathrm{Bol}_{n-j-k-1}\setminus\lbrace\hat{0}\rbrace|\star \mathbb{S}^{k-1}\simeq\mathbb{S}^{j+k-2}\star | \mathrm{Bol}_{n-j-k-1}\setminus\lbrace\hat{0}\rbrace|$$
		Then, if $n-j-k-1=0$, $|A_{j,k}|\simeq\mathbb{S}^{j+k-2}\simeq\mathbb{S}^{n-3}$ and if $j\leq n-k-2$, $|A_{j,k}|$ is contractible.
				
		 If $j\geq n-k$ then each element of $\mathrm{Bol}_n$ can be seen as the disjoint union of elements of $\mathrm{Bol}_n^i$ for $i=1,2,3$, where $\mathrm{Bol}_n^1,\mathrm{Bol}_n^2$ and $\mathrm{Bol}_n^3$ denote the finite Boolean lattice of subsets of $P_0=\lbrace 1,\dots,n-k-1\rbrace,P_1=\lbrace n-k,\dots,j\rbrace$ and $P_2=\lbrace j+1,\dots,n\rbrace$ respectively. Since each $U\in A_{j,k}$ can be seen as the disjoint union $U=U_0\cup U_1\cup U_2$, where $U_i\in \mathrm{Bol}_n^i~\forall~i=1,2,3$ then we can rewrite $A_{j,k}$ as:
		$$A_{j,k}=\lbrace(U_0,U_1,U_2)\mid U_0\cup U_1\cup U_2\neq\emptyset,U_i\subseteq P_i\text{ for }i=0,1,2,U_0\cup U_1\neq P_0\cup P_1\text{ and }U_1\cup U_2\neq P_1\cup P_2\rbrace$$
		Define:
		$$A_{j,k}^1=\lbrace(U_0,U_1,U_2)\mid U_0\cup U_1\cup U_2\neq\emptyset,U_i\subseteq P_i\text{ for }i=0,2,\text{ and }U_1\subsetneq P_1\rbrace$$
		$$A_{j,k}^2=\lbrace(U_0,U_1,U_2)\mid U_0\cup U_1\cup U_2\neq\emptyset,U_1\subseteq P_1,U_0\subsetneq P_0\text{ and } U_2\subsetneq  P_2\rbrace$$
		Then, $A_{j,k}=A_{j,k}^1\cup A_{j,k}^2$ with $\abs{A_{j,i}^i}$ contractible for $i=1,2$ by Lemma \ref{Lemma 9.4}. Hence, we obtain the following pushout diagram:
		$$ \begin{tikzcd}
			\abs{A_{i,j}^1\cap A_{i,j}^2} \arrow[hook]{r} \arrow[swap,hook]{d} & \abs{A^1_{i,j}}\simeq\bullet \arrow[hook]{d}\\%
			\abs{A^2_{i,j}}\simeq\bullet \arrow[hook]{r}& \abs{A_{i,j}}
		\end{tikzcd}$$
		which implies that $\abs{A_{i,j}}$ is homotopy equivalent to the suspension $\mathbb{S}^0\star\abs{A_{i,j}^1\cap A_{i,j}^2}$. Moreover
		$$\begin{aligned}&\abs{A_{i,j}^1\cap A_{i,j}^2}\simeq\abs{((\mathrm{Bol}_{n-k-1}\setminus\lbrace\hat{1}\rbrace)\times(\mathrm{Bol}_{j-n+k+1}\setminus\lbrace\hat{1}\rbrace)\times(\mathrm{Bol}_{n-j}\setminus\lbrace\hat{1}\rbrace))\setminus\lbrace\hat{0}\rbrace}\simeq\\
		&\qquad\simeq\abs{\mathrm{Bol}_{n-k-1}\setminus\lbrace\hat{0},\hat{1}\rbrace}\star \abs{\mathrm{Bol}_{j-n+k+1}\setminus\lbrace\hat{0},\hat{1}\rbrace}\star \abs{\mathrm{Bol}_{n-j}\setminus\lbrace\hat{0},\hat{1}\rbrace}\simeq\mathbb{S}^{n-k-3}\star\mathbb{S}^{j-n+k-1}\star\mathbb{S}^{n-j-2}\simeq\mathbb{S}^{n-4}\end{aligned}$$
		Hence, $\abs{A_{i,j}}\simeq\mathbb{S}^0\star\mathbb{S}^{n-4}\simeq\mathbb{S}^{n-3}$.\qedhere
\end{proof}

Now, we proceed with the computation of the Sigma invariants of affine Artin groups. We begin with the following lemma that will reduce the problem  to the study of a small family of characters.

\begin{lemma}\label{Lemma 9.3}
 Let $G=A_\Gamma$ be an irreducible affine Artin group with  $\abs{\Gamma}=n$ and $\chi:G\to\mathbb{R}$ a character such that for any $\Delta\in\mathcal{B}^\chi$,  $\chi(A_{\Delta})=0$. Then
\begin{itemize}
\item[i)] $[\chi]\in\Sigma^{n-2}(G,\mathbb{Z})$,
\item[ii)] $[\chi]\not\in\Sigma^{n-1}(G,\mathbb{Z})$.
\end{itemize}
\end{lemma}
\begin{proof} The hypothesis that $\chi$ vanishes on all the elements of $\mathcal{B}^\chi$ implies that for any $\Delta\in\mathcal{B}^\chi$, if $\Delta$ is not maximal, then $L_\Delta^\chi$ is precisely the flag complex associated to the living subgraph $\mathrm{Liv}^\chi$ which is strictly smaller than $\Gamma$ so  it is a complete spherical graph. This means that  for $k=\abs{\mathrm{Liv}^\chi}$,
$$L_\Delta^\chi=\mathrm{Bol}_k\setminus\{\hat{0}\}$$
which is contractible. Assume now that $\Delta\in\mathcal{B}^\chi$ is maximal. Then $L_\Delta^\chi$ is no longer the whole flag complex of the living subgraph $\mathrm{Liv}^\chi$ but the result of removing the top cell of this flag complex. This means that 
$$L_\Delta^\chi\simeq\mathrm{Bol}_k\setminus\{\hat{0},\hat{1}\}\simeq\mathbb{S}^{k-2}$$
so $L_\Delta^\chi$ is $(k-3)=(n-2-|\Delta|-1)$-connected so the homotopical strong $(n-2)$-link condition holds and we get i). 

For ii), assume first that $\chi$ is discrete. Then, for $\Delta\in\mathcal{B}^\chi$ maximal 
$$\overline{H}_{n-1-|\Delta|-1}(L_\Delta^\chi)=\overline{H}_{k-2}(L_\Delta^\chi)\neq 0$$
so in the spectral sequence we have $E^1_{p,q}\neq 0$ for $p=|\Delta|$, $q=n-1-|\Delta|$. But, the argument at the beginning of the proof implies that for any $\Delta\neq\Delta_1\in\mathcal{B}^\chi$, $\overline{H}_{\ast}(L_{\Delta_1}^\chi)=0$ thus $E^1_{p_1,q_1}=0$ for $p_1=|\Delta_1|$ and $q_1=n-p_1-1$. This means that $E^1_{p,q}$ survives in the spectral sequence and $E^\infty_{p,q}\neq 0$ thus $\dim_{\mathbb{F}}H_{n-1}(\Ker(\chi),\mathbb{F})=\infty$ and we have ii). For the non discrete case, use that $\Sigma^{n-1}(G,\mathbb{Z})$ is open and the discrete characters are dense in $S(G)$.
\end{proof}
\begin{remark}
	In fact,  Lemma \ref{Lemma 9.3} applies to any Artin groups such that the defining graph $\Gamma$ is complete and every proper $\Delta\subsetneq\Gamma$ is spherical, such as triangle Artin groups.
\end{remark}

At this point we can compute the Sigma invariants of the affine Artin groups that have cyclic abelianization.
\begin{corollary}
	If $G=A_\Gamma$ is an irreducible affine Artin group with $S(G)=\mathbb{S}^0$ and $\abs{\Gamma}=n$ then $\Sigma^j(A_\Gamma,\mathbb{Z})=S(G)$ if $j\leq n-2$ and $\Sigma^j(A_\Gamma,\mathbb{Z})=\emptyset$ if $j\geq n-1$.
\end{corollary}
\begin{proof}
	It follows from Lemma \ref{Lemma 9.3}  and the fact that the only characters of $G$ are those sending all generators to $1$ or $-1$ because for all these characters, $\mathcal{B}^\chi=\{\emptyset\}$.
\end{proof}
Hence, we are left with two families of groups and three sporadic groups to study:
\begin{itemize}
	\item $G=A_{\widetilde{\mathbb{I}}_1}$: In this case $G$ is the free group on two generators, and it is well known that $\Sigma^1(G,\mathbb{Z})=\emptyset$.
	\item $G=A_{\widetilde{\mathbb{G}}_2}$: Applying Theorem \ref{Lemma 9.2} below we have that $\Sigma^1(G,\mathbb{Z})=S(G)$ and $\Sigma^j(G,\mathbb{Z})=0~\forall~j\geq 2$.
	\item $G=A_{\widetilde{\mathbb{F}}_4}$: The only characters for which Lemma \ref{Lemma 9.3} does not apply are 
	$$\lbrace\pm(1,-1),\pm(2,-1),\pm(1,-2),\pm(3,-1)\rbrace.$$ For all of them we can use Theorem \ref{Theorem 8.5}  to deduce that $[\chi]\in\Sigma^j(G,\mathbb{Z})$
	if and only if $\chi$ satisfies the strong $j$-link condition. In the case when $\chi$ is in $\lbrace\pm(2,-1),\pm(1,-2),\pm(3,-1)\rbrace$, $\mathcal{B}^\chi=\{\emptyset,\Delta\}$ where $|\Delta|=3$ for $\chi$ in  $\lbrace\pm(2,-1),\pm(1,-2)\rbrace$ and  $|\Delta|=4$ for $\chi$ in  $\lbrace\pm(3,-1)\rbrace$. In all these cases it is easy to see that $L_\emptyset^\chi$ is contractible using Lemma \ref{Lemma 9.4} because the poset of simplices of $L_\Delta^\chi$ is precisely $\{\emptyset\neq U\in\mathrm{Bol}_n\mid \Delta\not\subseteq U\}$. And $L_\Delta^\chi=\mathrm{Bol}_{5-|\Delta|}\setminus\{\hat{0},\hat{1}\}=\mathrm{S}^{3-|\Delta|}$. This means that  $[\chi]\in\Sigma^3(G)$ but $[\chi]\not\in\Sigma^4(G,\mathbb{Z})$.
	
	When $\chi$ is one of $\pm(1,-1)$,  $\mathcal{B}^\chi=\{\emptyset,\Delta_1,\Delta_2\}$ with $\Delta_1\subseteq\Delta_2$, $|\Delta_1|=2$ and $|\Delta_2|=4$. In the same way, Lemma \ref{Lemma 9.4} implies that $L_\emptyset^\chi$ and $L_{\Delta_1}^\chi$ are contractible. This implies that $[\chi]\in\Sigma^3(G)$. Finally, the fact that $L_{\Delta_2}^\chi$ is empty implies that $[\chi]\not\in\Sigma^4(G,\mathbb{Z})$.

	\item $G=A_{\widetilde{\mathbb{B}_n}}$: The only characters for which Lemma \ref{Lemma 9.3} does not apply are 
$$\lbrace\pm(j,-1),1\leq j\leq n-2\rbrace.$$
For all of them, Theorem \ref{Theorem 8.5} gives a characterization in terms of the strong link condition. Consider first the case when $j\leq n-3$. Then $\mathcal{B}^\chi=\{\emptyset,\Delta\}$ with $\abs{\Delta}=j+1$. Therefore, Lemma \ref{Lemma 9.4} implies that $L_\emptyset^\chi$ is contractible and $L_\Delta^\chi\simeq\mathrm{Bol}_{n-\abs{\Delta}}\setminus\{\hat{0},\hat{1}\}\simeq\mathbb{S}^{n-2-\abs{\Delta}}$ so $\chi$ satisfies the strong homotopical $(n-2)$-link condition but not the  strong homological $(n-1)$-link condition which means that $[\chi]\in\Sigma^{n-2}(G)$, $[\chi]\not\in\Sigma^{n-1}(G,\mathbb{Z})$.

Let now $\chi=\pm(n-2,-1)$. We have $\mathcal{B}^\chi=\{\emptyset,\Delta_1,\Delta_2\}$ with $\abs{\Delta_1}=\abs{\Delta_2}=n-1$ and the poset of simplices of
$L_\emptyset^\chi$ is:
$$\{\emptyset\neq U\in\mathrm{Bol}_n\mid \Delta_1,\Delta_2\not\subseteq U\}=A_{4,3}.$$
where we are using the notation of Proposition \ref{Proposition 9.5}. Then, $L_\emptyset^\chi$ is homotopic to $\mathbb{S}^{n-3}$. This means that $\chi$ satisfies the strong homotopical  $(n-3)$-link condition so $[\chi]\in\Sigma^{n-3}(G)$. But $L_{\Delta_1}^\chi=L_{\Delta_2}^\chi=\emptyset$ so $\chi$ does not satisfy the strong homological  $(n-2)$-link condition so $[\chi]\not\in\Sigma^{n-2}(G,\mathbb{Z})$.

	\item $G=A_{\widetilde{\mathbb{C}_n}}$: The only characters for which Lemma \ref{Lemma 9.3} does not apply are $\Omega_1\cup\Omega_2$ with
$$\Omega_1=\lbrace\pm(j,-1,i)\mid 1\leq i\leq n-2, j\in\mathbb{Z}\rbrace,$$
$$\Omega_2=\lbrace\pm(i,-1,j)\mid 1\leq i\leq n-2, j\in\mathbb{Z}\rbrace,$$
By symmetry we may assume $\chi\in\Omega_1$. We distinguish two cases. Assume first that $j\not\in\{0,\ldots,n-2\}$. Then $\mathcal{B}^\chi=\{\emptyset,\Delta\}$ for $\abs{\Delta}=j+1$. Using Lemma \ref{Lemma 9.4}
one sees that  $L_\emptyset^\chi$ is contractible and $L_\Delta^\chi=\mathrm{Bol}_{n-\abs{\Delta}}\setminus\{\hat{0},\hat{1}\}\simeq\mathbb{S}^{n-2-\abs{\Delta}}$. This, together with  Theorem \ref{Theorem  8.5}, implies $[\chi]\in\Sigma^{n-2}(G)$, $[\chi]\not\in\Sigma^{n-1}(G,\mathbb{Z})$.

Assume now that $0\leq i\leq n-2$. Then, if $j+i+2<n$, $\mathcal{B}^\chi=\{\emptyset,\Delta_1,\Delta_2,\Delta_3\}$ for $\Delta_3=\Delta_1\cup\Delta_2$ and in other case $\mathcal{B}^\chi=\{\emptyset,\Delta_1,\Delta_2\}$ where $\abs{\Delta_1}=j+1$ and $\abs{\Delta_2}=i+1$.  

With the notation of Proposition \ref{Proposition 9.5}, we see that $L_\emptyset^\chi=A_{j+1,i}$ is either contractible or homotopic to $\mathbb{S}^{n-3}$, and this happens precisely when $j+i+1> n-1$. 
Consider now $L_{\Delta_2}^\chi$, which is either contractible by Lemma \ref{Lemma 9.4} or satisfies $L_{\Delta_2}^\chi=\mathrm{Bol}_{n-\abs{\Delta_2}}\setminus\{\hat{0},\hat{1}\}=\mathrm{S}^{n-2-\abs{\Delta_2}}$. The same dichotomy happens for $L_{\Delta_1}^\chi$.

 Finally, in the case when $j+i+2<n$, we have $L_{\Delta_3}^\chi=\mathrm{Bol}_{n-\abs{\Delta_3}}\setminus\{\hat{0},\hat{1}\}=\mathrm{S}^{n-2-\abs{\Delta_3}}$. Therefore, if  $j+i+2<n$, $\chi$ satisfies the strong homotopical $(n-2)$-link condition but not the strong homological $(n-1)$-link condition while, if $j+i+1> n-1$, $\chi$ satisfies the strong homotopical $(n-3)$-link condition but not the strong homological $(n-2)$-link condition.
 
  Note that we can apply Theorem \ref{Theorem  8.5} in all these cases except when $i=0$, i.e., for $\chi=\pm(j,-1,0)$ for $1\leq j\leq n-2$. But in these cases we can apply Remark \ref{after8.5}. So we get that $[\chi]\in\Sigma^{n-2}(G)$, $[\chi]\not\in\Sigma^{n-1}(G,\mathbb{Z})$ if  $j+i+2\leq n$ and $[\chi]\in\Sigma^{n-3}(G)$, $[\chi]\not\in\Sigma^{n-2}(G,\mathbb{Z})$ otherwise.
\end{itemize}

Then, the $\Sigma$-invariants of the irreducible affine Artin groups are the following:
\begin{theorem} \begin{itemize}
		\item[i)] If $G\in\lbrace A_{\widetilde{\mathbb{A}}_n},A_{\widetilde{\mathbb{D}}_n},A_{\widetilde{\mathbb{E}}_6},A_{\widetilde{\mathbb{E}}_7},A_{\widetilde{\mathbb{E}}_8},A_{\widetilde{\mathbb{F}}_4},A_{\widetilde{\mathbb{G}}_2},A_{\widetilde{\mathbb{I}}_1}\rbrace$ then: 
		$$\Sigma^j(G)=\Sigma^j(G,\mathbb{Z})=S(G)\text{ if }j\leq\abs{\Gamma}-2,$$
		$$\Sigma^j(G)=\Sigma^j(G,\mathbb{Z})=\emptyset\text{ if }j\geq\abs{\Gamma}-1.$$
		\item[ii)] If $G=A_{\widetilde{\mathbb{B}}_n}$ then:
		$$\Sigma^j(G)=\Sigma^j(G,\mathbb{Z})=S(G)\text{ if }j\leq n-3,$$
		$$\Sigma^j(G)=\Sigma^j(G,\mathbb{Z})=\emptyset\text{ if }j\geq n-1,$$
		$$\Sigma^{n-2}(G)=\Sigma^{n-2}(G,\mathbb{Z})=S(G)\setminus\lbrace\pm[(n-2,-1)]\rbrace.$$
		\item[iii)] If $G=A_{\widetilde{\mathbb{C}}_n}$ then:
		$$\Sigma^j(G)=\Sigma^j(G,\mathbb{Z})=S(G)\text{ if }j\leq n-3,$$
		$$\Sigma^j(G)=\Sigma^j(G,\mathbb{Z})=\emptyset\text{ if }j\geq n-1,$$
		$$\Sigma^{n-2}(G)=\Sigma^{n-2}(G,\mathbb{Z})=S(G)\setminus\lbrace\pm[(j,-1,i)]\mid j+i+2\geq n\rbrace.$$
		\end{itemize}
\end{theorem}
From  this result and Theorem \ref{Theorem 2.3} we have the following corollary:
\begin{corollary}
	Let $G$ be a spherical affine Artin group. Then:
	\begin{enumerate}
		\item If $G\in\lbrace A_{\widetilde{\mathbb{A}}_n},A_{\widetilde{\mathbb{D}}_n},A_{\widetilde{\mathbb{E}}_6},A_{\widetilde{\mathbb{E}}_7},A_{\widetilde{\mathbb{E}}_8},A_{\widetilde{\mathbb{F}}_4},A_{\widetilde{\mathbb{G}}_2},A_{\widetilde{\mathbb{I}}_1}\rbrace$ then $G'$ is $F_{n-2}$ but not $FP_{n-1}$, where $n$ is the number of standard generators of the Artin group.
		\item If $G\in\lbrace A_{\widetilde{\mathbb{B}}_n},A_{\widetilde{\mathbb{C}}_n}\rbrace$ then $G'$ is $F_{n-3}$ but not $FP_{n-2}$.
	\end{enumerate}
\end{corollary}
\subsection{Triangle  Artin groups}

A triangle Artin group is an Artin group whose defining graph is a triangle. If the triangle has labels $M,N$ and $P$, we will denote this group by $A_{MNP}$.  Triangle Artin groups satisfy  the  $\Sigma^1$-conjecture  (\cite{Almeida-Kochloukova}) and are 2-dimensional unless they are spherical which happens precisely  if $\{M,N,P\}=\{2,2,P\},\{2,3,3\},\{2,3,4\},\{2,3,5\}$. 

\begin{lemma}\label{Lemma 9.2} If $A_{MNP}$ is a non-spherical triangle Artin group, then $\Sigma^2(A_{MNP})=\Sigma^2(A_{MNP},\mathbb{Z})=\emptyset$ and $[\chi]\in\Sigma^1(A_\Gamma)$ if and only if $\text{Liv}^\chi$ is connected.
\end{lemma}
\begin{proof} The last assertion follows from the fact that for these groups the $\Sigma^1$-conjecture is true so it suffices to show that $\Sigma^2(A_{MNP},\mathbb{Z})=\emptyset$. To do that using the facts that discrete characters are dense and the $\Sigma$-invariants are open it suffices to prove that no discrete character $\chi$ lies in $\Sigma^2(A_{MNP},\mathbb{Z})$. Moreover, we may assume that $\text{Liv}^\chi$ is connected because in other case $[\chi]\not\in\Sigma^1(A_{MNP})$.
Assume first that   $\chi$ vanishes in all the subgraphs $\Delta\in\mathcal{B}^\chi$. Then, Lemma \ref{Lemma 9.3} and the remark right afterwards imply that $[\chi]\not\in\Sigma^2(A_{MNP})$. So we may assume that there is an edge $(v,u)\in\mathcal{B}^\chi$ such that $\chi(u),\chi(v)\neq 0$. As we are assuming that $\text{Liv}^\chi$ is connected, the third vertex $w$ can not be in $\mathcal{B}^\chi$ either. Therefore $\mathcal{B}^\chi=\{\emptyset,(v,u)\}$ and $L_\emptyset^\chi$ is the defining graph with the open edge $(v,u)$ deleted so it is contractible. Besides, $L_{(v,u)}^\chi$ consists of two isolated points thus, using the spectral sequence of Section \ref{Section 6.3}, we have in page one only one non zero term which is $E^1_{2,0}=\overline{H}_0(L_{(v,u)}^\chi)$. This term survives in $E^\infty_{2,0}$ so $\dim_{\mathbb{F}}H_{2}(\Ker(\chi),\mathbb{F})=\infty$ and $[\chi]\not\in\Sigma^2(A_{MNP},\mathbb{Z})$.\end{proof}

In particular, for any triangle Artin group the character $\chi:A_{MNP}\to\mathbb{Z}$ that maps all standard generators to $1\in\mathbb{Z}$ lies in $\Sigma^1$ so in the non spherical case $\Sigma^1\neq\Sigma^2$. This algebraic argument proves that non-spherical triangle Artin groups are not coherent because for coherent Artin groups $\Sigma^1$ coincides with $\Sigma^2$. Recall that a group $G$ is {\bf coherent} if it has the property that any finitely generated subgroup is also finitely presented. In \cite{Wise} there is a geometric proof that $A_{235}$ is not coherent, the  argument there also applies to the spherical triangle groups $A_{233}$ and $A_{234}$. Hence, the groups $A_{22P}$ are the only coherent triangle Artin groups.
\bibliographystyle{unsrt}

\end{document}